\def\myfontsize{10pt}
\documentclass[\myfontsize]{amsart}
\usepackage{amsmath, amsfonts, amssymb}
\usepackage{stmaryrd}
\usepackage{accents}
\usepackage[all]{xy}
\usepackage{graphicx}
\usepackage{hhline}
\usepackage{tikz}
\usetikzlibrary{patterns,calc}
\usepackage[colorlinks=true, linkcolor=blue,
  citecolor=blue, urlcolor=blue]{hyperref}
\ifcsname mypreamble\endcsname\mypreamble\fi
\usepackage{dynkin-diagrams}
\usepackage{tikz-cd}

\def\newthm#1#2{\newtheorem{#1}[dummy]{#2}%
  \expandafter\def\csname#2\endcsname##1{\hyperref[#1:##1]{#2~\ref*{#1:##1}}}}
\newthm{thm}{Theorem}
\newthm{lemma}{Lemma}
\newthm{prop}{Proposition}
\newthm{cor}{Corollary}
\newthm{conj}{Conjecture}
\newthm{cons}{Consequence}
\newthm{claim}{Claim}
\newthm{fact}{Fact}
\newthm{obs}{Observation}
\newthm{guess}{Guess}
\theoremstyle{definition}
\newthm{defn}{Definition}
\newthm{remark}{Remark}
\newthm{example}{Example}
\newthm{question}{Question}
\newthm{notation}{Notation}
\newthm{problem}{Problem}

\newcommand{\Section}[1]{\hyperref[sec:#1]{Section~\ref*{sec:#1}}}
\newcommand{\Table}[1]{\hyperref[tab:#1]{Table~\ref*{tab:#1}}}
\newcommand{\Figure}[1]{\hyperref[fig:#1]{Figure~\ref*{fig:#1}}}
\newcommand{\eqn}[1]{\hyperref[eqn:#1]{(\ref*{eqn:#1})}}

\DeclareMathOperator{\GL}{GL}
\DeclareMathOperator{\SL}{SL}
\DeclareMathOperator{\PGL}{PGL}
\DeclareMathOperator{\Sp}{Sp}
\DeclareMathOperator{\SO}{SO}
\DeclareMathOperator{\PSO}{PSO}
\DeclareMathOperator{\Gr}{Gr}

\DeclareMathOperator{\LG}{LG}
\DeclareMathOperator{\IG}{IG}

\DeclareMathOperator{\OG}{OG}

\DeclareMathOperator{\Hom}{Hom}
\DeclareMathOperator{\Aut}{Aut}
\DeclareMathOperator{\Ker}{Ker}
\DeclareMathOperator{\Sym}{Sym}

\DeclareMathOperator{\Hilb}{Hilb}
\DeclareMathOperator{\HH}{H}

\DeclareMathOperator{\Stab}{Stab}

\DeclareMathOperator{\codim}{codim}
\DeclareMathOperator{\rank}{rk}

\newcommand{\bG}{{\mathbb G}}

\newcommand{\bP}{{\mathbb P}}

\newcommand{\C}{{\mathbb C}}

\newcommand{\Q}{{\mathbb Q}}
\newcommand{\Z}{{\mathbb Z}}

\newcommand{\cC}{{\mathcal C}}

\newcommand{\cE}{{\mathcal E}}
\newcommand{\cF}{{\mathcal F}}
\newcommand{\cH}{{\mathcal H}}
\newcommand{\cI}{{\mathcal I}}
\newcommand{\cK}{{\mathcal K}}

\newcommand{\cO}{{\mathcal O}}
\newcommand{\cP}{{\mathcal P}}

\newcommand{\fg}{{\mathfrak g}}
\newcommand{\fp}{{\mathfrak p}}

\newcommand{\fq}{{\mathfrak q}}
\newcommand{\fl}{{\mathfrak l}}

\newcommand{\fz}{{\mathfrak z}}
\newcommand{\fr}{{\mathfrak r}}
\newcommand\SLT[1]{{S_{#1}}}
\newcommand\Stt{{{\texttt{S}}}}
\newcommand\Ssf{{{\textsf{S}}}}
\newcommand\Rsf{{{\textsf{R}}}}
\newcommand\Qsf{{{\textsf{Q}}}}

\newcommand{\ev}{\operatorname{ev}}

\newcommand{\ov}{\overline}

\newcommand{\ignore}[1]{}

\newcommand{\scal}[1]{\langle #1 \rangle}

\begin{document}

\title{Sphericality of spaces of conics \\
on homogeneous spaces}

% Revision information from git. The variables \gitdate and \gittag are
% set by the shell script mkpdf.sh
\ifdefined\gitdate
\date{\gitdate\ revision {\tt \gittag}}
\else
\date{June 16, 2026}
\fi

\author{Minseong Kwon}
\address{Morningside Center of Mathematics, Academy of Mathematics and Systems Science, Chinese Academy of Sciences, Beijing 100190, China}
\email{minseong@amss.ac.cn}

\author{Nicolas Perrin}
\address{Centre de Math\'ematiques Laurent Schwartz (CMLS), CNRS, \'Ecole
polytechnique, Institut Polytechnique de Paris, 91120 Palaiseau, France}
\email{nicolas.perrin.cmls@polytechnique.edu}



\thanks{M.K. was partially supported by the Institute for Basic Science (IBS-R032-D1) and by Morningside Center of Mathematics, Chinese Academy of Sciences.
N.P. was partially supported by ANR project FanoHK, grant ANR-20-CE40-0023.}

\begin{abstract}
  We prove that the Hilbert scheme of conics on $X = G/P$, a projective rational homogeneous space with Picard rank one, is $G$-spherical if and only if $P$ is associated to a long root of the Dynkin diagram.
\end{abstract}

\maketitle

%%%%%%%%%%%%%%%%%%%%%%%%%%%%%%%%%%%%%%%%%%%%%%%%%%%%%% newintro

\section{Introduction}

\subsection{Motivation}
We work over $\C$, the field of complex numbers.
Rational curves have been widely used in the study of Fano manifolds. 
Rational homogeneous spaces form an important class of Fano manifolds, and the geometry of lines on rational homogeneous spaces is well understood.
Various characteristic properties of rational homogeneous spaces of Picard rank one have been revealed based on a good understanding of the space of lines.
Notable examples include: characterisation of rational homogeneous spaces via their spaces of lines (\cite{HH:characterization}); deformation rigidity of rational homogeneous spaces (\cite{HwMo:Rigidity}); characterisation of surjective morphisms from rational homogeneous spaces of Picard rank one (answer to the so-called Lazarsfeld conjecture; see \cite[\S4.5]{hwang:geometry}).
We refer to \cite{LM:representation} for a further discussion.

It is natural to expect that the study of rational curves of higher degree should provide a useful tool.
However, spaces of rational curves of higher degree have not been fully described, even in the case of conics, i.e., rational curves of degree two.
The main difficulty is that the space of smooth rational curves is no more compact if the degree is at least two, and hence one need to consider its compactification.
A natural compactification is the Hilbert scheme of conics, and in this paper, we investigate the Hilbert scheme of conics on a rational homogeneous space.
Our main theorem states that for a rational homogeneous space of Picard rank one, the Hilbert scheme of conics is a spherical variety with only few exceptions.

\subsection{Statement of main theorem}

Let $X$ be a rational homogeneous space of Picard rank one.
Then $X = G/P$ for a simple Lie group $G$ and a maximal parabolic subgroup $P \subset G$.
Denote by $\varpi$ the fundamental weight defining $P$ and by $\delta$ the simple root with $\scal{\delta^\vee,\varpi} = 1$.
Define $\Hilb_{d}(X) \coloneqq \Hilb_{dt+1}(X \subset \bP(V_{\varpi}))$, the Hilbert scheme of closed subschemes with Hilbert polynomial $dt+1$.
The structure of $\Hilb_{1}(X)$, the Hilbert scheme of lines on $X$, is well understood.

\begin{thm}[\cite{landsberg.manivel:on}]
\label{thm:landsberg-manivel}
$\Hilb_{1}(X)$ consists of at most two $G$-orbits.
Moreover, $\Hilb_{1}(X)$ is $G$-homogeneous if and only if $\delta$ is a long simple root.
\end{thm}

Indeed, it turns out that, whether $\delta$ is long or not, $\Hilb_{1}(X)$ is a $G$-spherical variety (Theorem~\ref{thm:small-d=1}), that is, it is a normal $G$-variety containing an open dense $B$-orbit where $B$ is a Borel subgroup of $G$.
We refer to \cite{timashev:homogeneous} and \cite{perrin:geometry} for more on spherical varieties.

It is known in the literature that for some $X$, the Hilbert scheme $\Hilb_{2}(X)$ of conics is also $G$-spherical: see Corniani--Massarenti \cite{corniani.massarenti:complete} when $X = \LG(2,\,4)$, and the first author \cite{kwon:spherical} when $X$ is an adjoint variety not of type $A$ nor $C$.
On the other hand, in contrast to the case of lines, $\Hilb_{2}(X)$ is not always $G$-spherical, as shown in the following examples.

\begin{example}
\label{exam:F4/P3}
We list some cases for which $\Hilb_2(X)$ is not $G$-spherical for dimension reasons.
\begin{enumerate}
\item Let $X = \Sp_{2n}/P_1$ with notations as in \cite{bourbaki:elements*78}. Then $X = \bP^{2n-1}$, $\dim(X) = 2n-1$ and $c_1(X) = 2n$. The dimension of $\Hilb_2(X)$ is equal to $6n-4$ while the dimension of a Borel subgroup in $\Sp_{2n}$ is equal to $n(n+1)$. In particular $\Hilb_2(X)$ is not $\Sp_{2n}$-spherical for $n = 2$ or $n = 3$. Actually, we will see in Propositions \ref{prop:conic-sp} and \ref{prop:near-short} that $\Hilb_2(X)$ is not  $\Sp_{2n}$-spherical for all $n \geq 3$.
\item Let $X = \SO_{2n+1}/P_n$ with notations as in \cite{bourbaki:elements*78}. Then $\dim(X) = \frac{1}{2}(n^{2}+n)$ and $c_1(X) = 2n$. The dimension of $\Hilb_2(X)$ is equal to $\frac{1}{2}(n^{2}+9n-6)$ while the dimension of a Borel subgroup in $\SO_{2n+1}$ is equal to $n(n+1)$. In particular $\Hilb_2(X)$ is not $\SO_{2n+1}$-spherical for $n \in [2,\,5]$. We will see in Propositions \ref{prop:geom-B} and \ref{prop:near-long} that $\Hilb_2(X)$ is not  $\SO_{2n+1}$-spherical for all $n \geq 4$.
\item Let $X = F_4/P_3$ with notations as in \cite{bourbaki:elements*78}. Then $\dim(X) = 20$ and $c_1(X) = 7$. The dimension of $\Hilb_2(X)$ is equal to $31$ while the dimension of a Borel subgroup in $F_4$ is equal to $28$. In particular $\Hilb_2(X)$ is not $F_4$-spherical.
\item Let $X = F_4/P_4$ with notations as in \cite{bourbaki:elements*78}. Then $\dim(X) = 15$ and $c_1(X) = 11$. The dimension of $\Hilb_2(X)$ is equal to $34$ while the dimension of a Borel subgroup in $F_4$ is equal to $28$. In particular $\Hilb_2(X)$ is not $F_4$-spherical.
\item Let $X = G_2/P_1$ with notations as in \cite{bourbaki:elements*78}. Then $\dim(X) = 5$ and $c_1(X) = 5$. The dimension of $\Hilb_2(X)$ is equal to $12$ while the dimension of a Borel subgroup in $G_2$ is equal to $8$. In particular $\Hilb_2(X)$ is not $G_2$-spherical.
\end{enumerate}
\end{example}

Now a natural question arises: for which $X$ is $\Hilb_{2}(X)$ $G$-spherical?
Our main theorem gives a complete answer to this question.

\begin{prop}%[Proposition~\ref{prop:Hilb-irreducible} and Corollary~\ref{cor:hilb-is-smooth}]
    $\Hilb_{2}(X)$ is a smooth irreducible variety.
\end{prop}

\begin{thm} \label{thm:main}
    $\Hilb_{2}(X)$ is $G$-spherical if and only if $\delta$ is a long simple root.
\end{thm}

Notice that given a normal $G$-variety, being $G$-spherical only depends on the open $G$-orbit.
Therefore the $G$-sphericality still holds true for other (normalized) compactified spaces of conics on $X$ where the natural $G$-action extends such as the Chow variety or the space of unmarked stable maps, provided that $\delta$ is long.

It is worth mentioning that the length condition in Theorem~\ref{thm:main} is same with that in Theorem~\ref{thm:landsberg-manivel}.
Thus we conclude that $\Hilb_{1}(X)$ is $G$-homogeneous if and only if $\Hilb_{2}(X)$ is $G$-spherical.

\subsection{Some consequences}
\label{subsection:consequences}

Theorem~\ref{thm:main} yields a number of consequences, provided that $X$ is associated to a long simple root.
Indeed, thanks to the well-known description of $\Aut(X)$, Theorem~\ref{thm:main} applies to all $X$ but few exceptions:

\begin{fact}
    $X = \Aut(X)/P_{X}$ where $P_{X}$ is associated to a long simple root if and only if $X$ is neither $\Sp_{2n}/P_{i}$ ($2 \le i \le n-1$) nor $F_{4}/P_{i}$ ($i \in \{3,\,4\}$) with notations as in \cite{bourbaki:elements*78}.
\end{fact}

%To introduce corollaries of Theorem~\ref{thm:main}, (only) in this section, we use the following notation:

\begin{notation} \label{notation:Aut-long}
    $X$ is a rational homogeneous space of Picard rank one, neither $\Sp_{2n}/P_{i}$ ($2 \le i \le n-1$) nor $F_{4}/P_{i}$ ($i \in \{3,\,4\}$) with notations as in \cite{bourbaki:elements*78}.
\end{notation}

In the rest of Subsection \ref{subsection:consequences}, we assume that we are in the situation of Notation~\ref{notation:Aut-long}. Theorem~\ref{thm:main} reads as follows:

\begin{cor} \label{cor:Aut-spherical-Hilb2}
In Notation~\ref{notation:Aut-long}, $\Hilb_2(X)$ is $\Aut(X)$-spherical.
\end{cor}

Recall that B\"arligea (\cite{barligea:quasi-homogeneityI,barligea:quasi-homogeneityII}) proved existence of open dense $\Aut(X)$-orbit in $\overline{M}_{0,3}(X,d)$, the Kontsevich moduli space of genus zero stable maps with 3 marked points on $X$ (see \cite{fulton.pandharipande:notes}), for small $d$.
In particular, $\Hilb_{2}(X)$, which is birational to $\overline{M}_{0,0}(X,2)$, contains an open dense $\Aut(X)$-orbit.
In fact, since any $G$-spherical variety consists of only finitely many $B$-orbits, Theorem~\ref{thm:main} implies a stronger statement:
\begin{cor} \label{coro:finite num of orbits}
    In Notation~\ref{notation:Aut-long}, $\Hilb_{2}(X)$ has only finitely many $\Aut(X)$-orbits.
\end{cor}

Finally, we record another corollary of Theorem~\ref{thm:main}, which is of practical importance from the birational geometric point of view.

\begin{cor}\label{cor:MDS}
    In Notation~\ref{notation:Aut-long}, $\Hilb_{2}(X)$ is a Mori dream space.
\end{cor}

Some particular cases of Corollary~\ref{cor:MDS} are known in the literature: see Chen--Coskun \cite{chen.coskun} when $X$ is a projective space of small dimension, and Corniani--Massarenti \cite{corniani.massarenti:complete} when $X$ is a Lagrangian Grassmannian.
Corollary~\ref{cor:MDS} follows from Theorem~\ref{thm:main} together with the well-known fact that any $\Q$-factorial spherical variety is a Mori dream space, as shown by Brion \cite{brion:Mori}.

\subsection{Outline of two proofs}

We present two independent proofs of Theorem~\ref{thm:main}.
A geometric proof based on the fact that we may replace conics by bigger quadrics in $X$ and obtain a fibration of the open $G$-orbit in $\Hilb_2(X)$ to a smaller homogeneous space (see Sections \ref{section:curves}--\ref{section:geometric}). More precisely, for $x,y \in X$, set $d(x,y) = \min\{d \ | \ \textrm{$x$ and $y$ are on a degree $d$ rational curve}\}$ and define $\Gamma_d(x,y)$ as the locus covered by degree $d$ rational curves passing through $x$ an $y$. B\"arligea's Theorem implies that $\Aut(X)$ acts transitively on pairs of elements $x,y \in X$ general with $d(x,y) = 2$. In particular for such $x,y$ the variety $\Gamma_2(x,y)$ does not depend on $x$ and $y$ general . We denote by $\Gamma_2$ this variety (see Section \ref{section:curves} for an alternative definition of $\Gamma_2$). In particular, we prove the following result.

\begin{prop}
Assume that $X$ is not a projective space.
\begin{enumerate}
\item A general conic is contained in a unique translate of $\Gamma_2$.
\item $\Gamma_2$ is a smooth quadric maximal for the inclusion.
\item $\Stab_G(\Gamma_2)$ is always connected and spherical except for $X = F_4/P_3$.
\end{enumerate}
\end{prop}

\begin{remark}
The quadric $\Gamma_2$ is \emph{not} a quadric of maximal dimension contained in $X$ in general (see Example \ref{example:e6p4}) but there is no quadric $Q$ with $\Gamma_2 \subsetneq Q \subset X$. 
%It might still be a maximal quadric.
\end{remark}

\begin{remark}
For $X$ a projective space, we have $\Gamma_2 = X$ and $\Stab_G(\Gamma_2) = G$.
\end{remark}

Our stategy is then to pick a general conic $C_2$ contained in $\Gamma_2$ and to consider the map $G.C_2 \to G.\Gamma_2$. The source of this map is an open subset of $\Hilb_2(X)$ and the fiber over $\Gamma_2$ is an open subset of $\Hilb_2(\Gamma_2)$. The target is (almost) always spherical by the previous proposition. We then study the action of $\Stab_G(\Gamma_2)$ on $\Hilb_2(\Gamma_2)$ to conclude.

\begin{remark}
%The stabiliser $\Stab_G(\Gamma_2)$ is always connected (see Proposition \ref{prop:stab-sd} and AAA) however t
The stabiliser $\Stab_G(C_2)$ of a general conic is not always connected (see Remark \ref{remark:connected}).
\end{remark}

The second proof of Theorem \ref{thm:main} is based on representation theoretic computations of normal spaces of closed orbits in $\Hilb_{2}(X)$ (see Sections \ref{section:local}--\ref{section:representation}).
To be precise, consider a $B$-stable conic $C$ on $X$ and the $G$-orbit $G.C$ in $\Hilb_{2}(X)$.
Then a deformation theoretic argument shows that $\Hilb_{2}(X)$ is smooth (see Corollary \ref{cor:hilb-is-smooth}).
Moreover, by a local structure theorem for $G$-varieties and by Luna's \'etale slice theorem, the normal space $N_{C} := T_{C}\Hilb_{2}(X) / T_{C} (G . C)$, equipped with the isotropy representation, is spherical if and only if $\Hilb_{2}(X)$ is $G$-spherical (see Lemma \ref{lemma:loc-str-spherical}).
We use this fact to prove Theorem \ref{thm:main}, by computing the representation $N_{C}$ case by case.
In the process of the proof, we obtain the following result.

\begin{prop}
If $\delta$ is long, then $\Hilb_2(X)$ is a $G$-spherical variety of rank at most $4$.
 \end{prop}

Actually for $\delta$ long, the rank of $\Hilb_{2}(X)$ as a $G$-spherical variety is 2, 3 or 4. We refer to Corollary \ref{cor:rank} for a precise statement.

\subsection*{Acknowledgement}
We would like to thank Michel Brion, Kiryong Chung and Laurent Manivel for helpful discussions.
The first author was supported by the Institute for Basic Science (IBS-R032-D1) and by Morningside Center of Mathematics, Chinese Academy of Sciences.
The second author was supported by the project FanoHK ANR-20-CE40-0023.

\setcounter{tocdepth}{1}

\tableofcontents

%%%%%%%%%%%%%%%%%%%%%%%%%%%%%%%%%%%%%%%%%%%%%%%%%%%%%%
%\input{notations}
\section{Notations}
\label{section:notations}

Let $X = G/P$ be a homogeneous space defined by a connected %, simply connected, 
semisimple complex Lie group $G$ and a parabolic subgroup $P$. Fix a maximal torus $T$ and a Borel subgroup $B$ such that $T \subset B \subset P \subset G$. Let $R$ be the associated root system with positive roots $R^+$ and simple roots $\Delta \subset R^+$. For $\alpha \in R$, let $\fg_\alpha \subset \fg$ be the eigenspace of weight $\alpha$ and $U_\alpha \subset G$ be the $1$-dimensional %unipotent 
subgroup with Lie algebra $\fg_\alpha$. Let $\SLT{\alpha} \subset G$ be the subgroup of generated by $U_\alpha$ and $U_{-\alpha}$.

Let $W = N_G(T)/T$ be the Weyl group of $(G,T)$, and let $W_P = N_P(T)/T$ be the Weyl group of $(P,T)$. Set $\Delta_P = \{ \beta\in \Delta \ | \ s_\beta \in W_P\}$. The group $W_P$ is generated by the simple reflections $s_\beta$ for $\beta \in \Delta_P$. Set $R_P = R \cap \Z\Delta_P$ and $R_P^+ = R^+ \cap \Z\Delta_P$. For $\alpha \in \Delta$, we write $\varpi_\alpha$ for the fundamental weight associated to $\alpha$ and defined by $\scal{\beta^\vee,\varpi_\alpha} = \delta_{\alpha,\beta}$ for $\alpha,\beta\in \Delta$. Write $\Theta$ for the highest root in $R$ and fix $( \ , \ )$ an invariant scalar product on the root space. This gives an identification between the root and coroot spaces so that the following holds: ${\alpha^\vee = \frac{2 \alpha}{(\alpha,\alpha)}}$.

Each element $w \in W$ defines a Schubert variety $X_w = \overline{Bw.P}$ in $X$ and an opposite Schubert variety $X^w = \overline{B^-w.P}$, where $B^-$ is the opposite Borel subgroup defined by $B \cap B^- = T$. If $w$ is the minimal length representative in its coset in $W/W_P$, then we have $\dim X_w = \codim X^w = \ell(w)$. We denote by $W^X$ the set of these minimal coset representatives.

We define the Hecke product on $W$ as the unique associative monoid product such that for any simple reflection $s_\beta \in W$, we have $w \cdot s_\beta = ws_\beta$ if $\ell(ws_\beta) > \ell(w)$ and $s_\beta \cdot w =  w$ otherwise, we refer to \cite{buch.mihalcea:curve} for more results on the Hecke product.

Given a positive root $\alpha$ with $s_\alpha \not\in W_P$, let $C_\alpha \subset X$ be the unique $T$-stable curve that contains the $T$-fixed points $1.P$ and $s_\alpha.P$. We have $C_\alpha = \SLT{\alpha}.P$ (see \cite[Example 1.3.4]{brion:lectures}). The homology group $\HH_2(X) := H_2(X,\Z)$ can be identified with the quotient $\Z\Delta^\vee/\Z\Delta_P^\vee$ where $\Z\Delta^\vee$ is the coroot lattice of $G$ and $\Z\Delta_P^\vee$ is the coroot lattice of $P$. The degree $d(\alpha) := [C_\alpha] \in \HH_2(X)$ is equal to the image of the coroot $\alpha^\vee$ in  $\Z\Delta^\vee/\Z\Delta_P^\vee$. There is a dual identification $H^2(X,\Z) = \Z\{\varpi_\alpha \ | \ \alpha \in \Delta \setminus \Delta_P\}$ such that the intersection pairing is given by $\scal{\alpha^\vee,\varpi_\beta}$.

Assume now that $P \supset B$ is a maximal parabolic subgroup. There exists a fundamental weight $\varpi$ such that $P$ is the closed subgroup of $G$ containing $T$ and the root subgroups $U_\alpha$ for $\alpha$ with $\scal{\alpha^\vee,\varpi} \geq 0$. Let $\delta$ be the simple root such that $\scal{\delta^\vee,\varpi} = 1$ and recall that $\Theta$ is the highest root of $R$. Under the identification $\HH_2(X) \simeq \Z$, the degree $d(\alpha)$ is the coefficient of $\delta^\vee$ in the expression of $\alpha^\vee$ in terms of simple coroots. Since $d(\Theta)$ will play an important role, we give the value of $\Theta^\vee$ in terms of simple coroots in Table \ref{table-theta}.

\begin{table}[ht]
%\begin{tabular}{c|c}
%Type of $G$ & $d(\Theta) = \textrm{coeff. of } \delta^\vee \textrm{ in } \Theta^\vee$ \\
%\hline
%$A_n$ & $\dynkin[edge length=.7cm,
%        labels*={1,1,1,1,1},]A{***.**}$ \\
%$B_n$ & $\dynkin[edge length=.7cm,
%        labels*={1,2,2,2,1},]B{***.**}$ \\
%$C_n$ & $\dynkin[edge length=.7cm,
%        labels*={1,1,1,1,1},]C{***.**}$ \\
%$D_n$ &$\dynkin[edge length=.7cm,
%        labels*={1,2,2,2,1,1},]D{***.***}$ \\
%$E_6$ & $\dynkin[edge length=.7cm,
%        labels*={1,2,2,3,2,1},]E6$\\
%\hline
%\end{tabular}
%\begin{tabular}{c|c}
%Type of $G$ & $d(\Theta) = \textrm{coeff. of } \delta^\vee \textrm{ in } \Theta^\vee$ \\
%\hline
%$E_7$ &  $\dynkin[edge length=.7cm,
%        labels*={1,2,2,3,2,1},]E6$ \\
%$E_8$ &  $\dynkin[edge length=.7cm,
%        labels*={1,2,2,3,2,1},]E6$ \\
%$F_4$ &  $\dynkin[edge length=.7cm,
%        labels*={2,3,2,1},]F4$ \\
%$G_2$ & $\dynkin[edge length=.7cm,
%        labels*={1,2},]G2$ \\
%\hline
%\end{tabular}
\begin{tabular}{c|c}
Type of $G$ & $d(\Theta) = \textrm{coeff. of } \delta^\vee \textrm{ in } \Theta^\vee$ \\
\hline
$A_n$ & $\dynkin[edge length=.7cm,
        labels*={1,1,1,1,1},]A{***.**}$ \\
$B_n$ & $\dynkin[edge length=.7cm,
        labels*={1,2,2,2,1},]B{***.**}$ \\
$C_n$ & $\dynkin[edge length=.7cm,
        labels*={1,1,1,1,1},]C{***.**}$ \\
$D_n$ &$\dynkin[edge length=.7cm,
        labels*={1,2,2,2,1,1},]D{***.***}$ \\
$E_6$ & $\dynkin[edge length=.7cm,
        labels*={1,2,2,3,2,1},]E6$\\
$E_7$ &  $\dynkin[edge length=.7cm,
        labels*={2,2,3,4,3,2,1},]E7$ \\
$E_8$ &  $\dynkin[edge length=.7cm,
        labels*={2,3,4,6,5,4,3,2},]E8$ \\
$F_4$ &  $\dynkin[edge length=.7cm,
        labels*={2,3,2,1},]F4$ \\
$G_2$ & $\dynkin[%reverse arrows,
edge length=.7cm,
        labels*={1,2},]G2$ \\
\hline
\end{tabular}

~\\
\centering
~\\
\caption{\label{table-theta} The value of $d(\Theta)$.}
\end{table}

Denote by $\Gr(p,n)$ the grassmannian of $p$-dimensional subspaces in a vector space of dimension $n$. For $n \neq 2p$, let $\OG(p,n)$ be the subvariety of $\Gr(p,n)$ of subspaces isotropic for a non-degenerate quadratic form. For $n = 2p$ the previous isotropic grassmannian has two isomorphic connected components and $\OG(p,2p)$ denotes one of these connected components. Let $\IG(p,2n)$ be the subvariety of $\Gr(p,2n)$ of subspaces isotropic for a symplectic form.  For $p = n$, set $\LG(n,2n) = \IG(n,2n)$,.

%%%%%%%%%%%%%%%%%%%%%%%%%%%%%%%%%%%%%%%%%%%%%%%%%%%%%%
%\input{curves}

\section{Curves neighborhoods}
\label{section:curves}

\subsection{First definitions} Let $d \in \HH_2(X)$ be an effective degree and let $\Omega \subset X$ be any subset. Define the \emph{curve neighborhood} $\Gamma_d^X(\Omega)$ of $\Omega$ as the closure in $X$ of the  locus covered by smooth rational curves of degree $d$ meeting $\Omega$. For $\Omega,\Omega' \subset X$ two subsets, define the \emph{double curve neighborhood} $\Gamma_d^X(\Omega,\Omega')$ as the closure in $X$ of the locus covered by smooth rational curves of degree $d$ meeting $\Omega$ and $\Omega'$. If each of $\Omega$ and $\Omega'$ consists of a unique point: $\Omega = \{x\}$ and $\Omega' = \{y\}$, we set $\Gamma_d^X(x,y) := \Gamma_d^X(\Omega,\Omega')$. When $X$ is clear from the context, we will use $\Gamma_d$ in place of $\Gamma_d^X$.

\begin{defn}
Let $d \in \HH_2(X)$ be an effective degree. A maximal root for $d$ is a root $\alpha \in R^+$ maximal for the property $d(\alpha) \leq d$. A greedy decomposition of $d$ is a sequence $(\alpha_1,\cdots,\alpha_r)$ such that $\alpha_1$ is a maximal root of $d$ and  $(\alpha_1,\cdots,\alpha_r)$ is a greedy decomposition of $d - d(\alpha_1)$. The empty set is the unique greedy decomposition of $0$.
\end{defn}

We will need only few results on greedy decompositions.
\begin{prop}
Let $d \in \HH_2(X)$ be a non-zero effective degree.
\begin{enumerate}
\item There is, up to reordering, a unique greedy decomposition of $d$.
\item The first terms in all greedy decompositions are the same.
\end{enumerate}
\end{prop}

\begin{proof}
For (1) see \cite[Section 4.2]{buch.mihalcea:curve} and see  \cite[Proposition 3.16]{barligea:curve} for (2).
\end{proof}

\begin{cor}
There is a unique maximal root $\gamma(d)$ of $d$.
\end{cor}

\begin{defn}
For $d \in \HH_2(X)$ effective, define $z_d^P \in W^P$ by the equation $z_d^Pw_P = s_{\alpha_1} \cdot \ldots \cdot s_{\alpha_r} \cdot w_P$, where $(\alpha_1,\cdots,\alpha_r)$ is a greedy decomposition of~$d$.
\end{defn}

The fact that $z_d^P$ is well defined follows from the uniqueness of the greedy decomposition and results in \cite[Section 4]{buch.mihalcea:curve}. The fact that $z_d^P$ lies in $W^P$ follows from \cite[Proposition 2.4(8)]{barligea:curve}. The importance of this definition comes from the following result.

\begin{thm}[see \cite{buch.mihalcea:curve}]
\label{thm:curve-neigh-1}
For  $d \in \HH_2(X)$ effective, we have $\Gamma_d(X_w) = X_{z_d^P\cdot w}$
\end{thm}

\subsection{Curve neighborhoods of points} The previous result in particular implies that the degree $d$ curve neighborhood of the $B$-fixed point in $X$ is a Schubert variety. Explicitly, we have $\Gamma_d(\{1.P\}) = X_{z_d^P}$. Our aim is to describe $\Gamma_d(x,y)$ the degree $d$ double curve neigborhoods of $x,y \in X$ two general points on a degree $d$ rational curve. Since $G$ acts with finitely many orbits on $X^2$ (with orbits indexed by $W^X$), the set of pairs of points on a degree $d$ rational curve contains a dense $G$-orbit. In particular, using the $G$-action, we have a unique $T$-stable choice for the pair $(x,y)$, namely $(x,y) = (1.P,z_d^P.P)$. Set 
$$\Gamma_d^X := \Gamma_d(1.P,z_d^P.P).$$
When $X$ is clear from the context, we will use $\Gamma_d$ in place of $\Gamma_d^X$.

%$x = 1.P$ and $y = z_d^P.P$. Set $\Gamma_d := \Gamma_d(\{x\},\{y\})$.

We give a description of $\Gamma_d$ for small degrees. Recall that $X$ has Picard rank $1$ \emph{i.e.} that $P$ is a maximal parabolic subgroup. Recall the definition of the simple root $\delta \in \Delta$ such that $\Delta_P = \Delta \setminus \{\delta\}$ and of $\varpi = \varpi_\delta$. We have an identification $\HH_2(X) \simeq \Z$ given by $d(\alpha) \mapsto\scal{\alpha^\vee,\varpi}$. For any root $\alpha \in R$, let $n_\delta(\alpha)$ be the coefficient of $\delta$ in the expansion of $\alpha$ as linear combination of simple roots. For the reader's convenience we give a proof of the following well known result.

\begin{lemma}
\label{lemm:n}
For $\alpha \in R$, we have $(\alpha,\alpha)\scal{\alpha^\vee,\varpi} = (\delta,\delta)n_\delta(\alpha)$. 
\end{lemma}

\begin{proof}
Write $\alpha = \sum_{\beta \in \Delta}n_\beta(\alpha) \beta$. Then we have
$$(\alpha,\alpha)\alpha^\vee = 2 \alpha = \sum_{\beta \in \Delta} n_\beta(\alpha) (\beta,\beta) \beta^\vee.$$
Evaluating with $\varpi$ gives the result.
\end{proof}

\begin{lemma}
\label{lemm:root}
Let $d \in \HH_2(X)$ effective and non trivial with $d \leq d(\Theta)$. Then there exists a long root $\gamma$ with $\scal{\gamma^\vee,\varpi} = d$.
\end{lemma}

\begin{proof}
Taking coroots, the result follows from the fact that if $R$ is an irreducible root system, $\delta \in R$ is a simple root and $\theta \in R$ is the highest short root, then there exists a short root $\gamma$ with $n_\delta(\gamma) = d$ for all $d \in [0,n_\delta(\theta)]$.
%If $d = d(\Theta)$, take $\gamma = \Theta$. This implies the result for type $C$ since $n_\delta(\Theta) = 1$ for any simple root $\delta$. If $G$ is simply laced the result follows from the fact that for any $d \in [0,n_\delta(\Theta)]$, there is a root $\gamma$ with $n_\delta(\gamma) = d$. For type $B$, since $d(\Theta) \leq 2$, we can take $\gamma = \Theta$ if $d = 2$ and $\gamma = \delta$ if $d= 1$ except for $\delta$ short in which case we can take $\gamma = \Theta$. Finally it is an easy check in type $F_4$.
\end{proof}

Recall that for $d \in \HH_2(X)$, the root $\gamma(d)$ is the unique maximal root of $d$.

\begin{prop}
Let $d \in \HH_2(X)$ such that $d \leq d(\Theta)$. Then
\begin{enumerate}
\item $d(\gamma(d)) = d$.
\item $\gamma(d)$ is a long root.
\end{enumerate}
\end{prop}

\begin{proof}
By Lemma \ref{lemm:root}, there exists a long root $\gamma$ such that $\scal{\gamma^\vee,\varpi} = d$. By maximality of $\gamma(d)$, we have $\gamma \leq \gamma(d)$. Using Lemma \ref{lemm:n}, we get 
$$d = \scal{\gamma^\vee,\varpi} = \frac{(\delta,\delta)}{(\gamma,\gamma)} n_\delta(\gamma) \leq \frac{(\delta,\delta)}{(\gamma,\gamma)} n_\delta(\gamma(d)) = \frac{(\gamma(d),\gamma(d))}{(\gamma,\gamma)} \scal{\gamma(d)^\vee,\varpi}.$$
By definition, we have $\scal{\gamma(d)^\vee,\varpi} \leq d$, thus $(\gamma,\gamma) \leq (\gamma(d),\gamma(d))$. Since $\gamma$ is long, so is $\gamma(d)$ and we have an equality in all above inequalities, proving the result.
\end{proof}

\begin{cor}
\label{cor:z_d=gamma}
For $d \in \HH_2(X)$ with $d \leq d(\Theta)$, we have $z_d^P.P = s_{\gamma(d)}.P$.
\end{cor}

\begin{proof}
Since $d(\gamma(d)) = d$, there is a unique greedy decomposition of $d$ given by the root $\gamma(d)$. This implies that $z_d^Pw_P = s_{\gamma(d)} \cdot w_P$ and $z_d^P.P = s_{\gamma(d)}.P$.
\end{proof}

For a root $\alpha \in R$, recall the definition of the curve $C_\alpha$ from Section \ref{section:notations}.

\begin{thm}
\label{thm:Gamma-small-d}
Let $d \in \HH_2(X)$ be effective with $d \leq d(\Theta)$. Then $\Gamma_d = C_{\gamma(d)}$.
\end{thm}

\begin{proof}
Consider the evaluation map $\ev : \overline{M}_{0,2}(X,d) \to X^2$. Then $(1.P,z_d^P.P)$ lies in the image of this map and the fiber $F$ is a projective $T$-stable variety. By \cite[Proposition 13.5]{borel:linear}, the fiber $F$ must contain at least $\dim(F) + 1$ fixed points for the action of $T$. It is therefore enough to prove that $C_{\gamma(d)}$ is the unique $T$-stable curve joining $1.P$ and $z_d^P.P$.

An irreducible $T$-stable curve is of the form $\SLT{\alpha}w.P$ for $\alpha$ a positive root and $w \in W$. This curve has two $T$-fixed points $w.P$ and $s_\alpha w.P$ and is of degree $\scal{w^{-1}(\alpha)^\vee,\varpi}$. If $(C_i)_{i \in [1,r]}$ is a degree $d$ chain of $T$-stable curves with $d_i := \deg(C_i)$ joining $1.P$ and $z_d^P.P$, then there exist positive roots $(\alpha_i)_{i \in [1,r]}$ such that $C_i = \SLT{\alpha_i}s_{\alpha_{i-1}} \cdots s_{\alpha_1}.P$, $d_i = \scal{s_{\alpha_1} \cdots s_{\alpha_{i-1}}(\alpha_i)^\vee,\varpi}$ and $z_d^P.P = s_{\alpha_r}  \cdots s_{\alpha_1}.P$. Computing the weights of these $T$-fixed points, we get the equality $s_{\gamma(d)}(\varpi) = z_d^P(\varpi) = s_{\alpha_r}  \cdots s_{\alpha_1}(\varpi)$. An easy induction gives $\varpi - s_{\alpha_r}  \cdots s_{\alpha_1}(\varpi) = \sum_{i = 1}^r d_i \alpha_i$ while $\varpi - s_{\gamma(d)}(\varpi) = d \gamma(d)$. We therefore get the equality
$$\gamma(d) = \sum_{i = 1}^r \frac{d_i}{d} \alpha_i.$$
Evaluating against the coroot $\gamma(d)^\vee$ and since $\gamma(d)$ is long, we get
$$2 = \sum_{i = 1}^r \frac{d_i}{d} \scal{\gamma(d)^\vee,\alpha_i} \leq 2 \sum_{i = 1}^r \frac{d_i}{d} = 2.$$
In particular we have $\scal{\gamma(d)^\vee,\alpha_i} = 2$ for all $i$ and thus $\alpha_i = \gamma(d)$ for all $i$ (recall that $\gamma(d)$ is long). This in turn implies that the chain of curves $(C_i)_{i \in [1,r]}$ is made of the unique curve $C_{\gamma(d)}$ proving the claim and the theorem.
\end{proof}

For $d = 2 \leq d(\Theta)$, this result implies that $\Gamma_2$ is a maximal quadric.

\begin{prop}
\label{prop:max-quad}
If $\Gamma_2$ is a quadric, then it is a maximal quadric.
\end{prop}

\begin{proof}
Assume that there exists a quadric $Q \subset X$ with $\Gamma_2 \subset Q$. Recall that $\Gamma_2 = \Gamma_2^X(x,y)$ with  $(x,y) = (1.P,z_2^P.P)$. In particular $x,y \in Q$ and, for any $z \in Q$, there exists a conic contained in $Q$ passing through $x$, $y$ and $z$. This implies $z \in \Gamma_2^X(x,y)$ by definition thus $Q \subset \Gamma_2$ proving the claim.
\end{proof}

\begin{cor}
\label{cor:max-quad}
If $d(\Theta) \geq 2$, then $\Gamma_2$ is a maximal quadric.
\end{cor}

\begin{proof}
By Theorem \ref{thm:Gamma-small-d}, $\Gamma_2 = C_{\gamma(2)}$ is a conic. Apply Proposition \ref{prop:max-quad}. 
\end{proof}

\begin{example}
\label{example:e6p4}
For $X = E_6/P_4$, the quadric $\Gamma_2$ is reduced to the conic $ C_{\gamma(2)}$ but is maximal quadric by Corollary \ref{cor:max-quad}. However, there are quadrics of larger dimension in $X$. Indeed, if $H \subset G$ is the subgroup generated by the root subgroups $U_\alpha$ for $\alpha \in \{\pm \alpha_i \ | \ i \in [3,5]\}$ where $\alpha_i$ are the simple roots with notations as in \cite{bourbaki:elements*78}, then $Y := H/(H \cap P) \subset X$ is a smooth quadric of dimension $4$. The family of such quadrics has dimension $30$. The family of conics contained in a translate of $Y$ has dimension 39 while the dimension of all quadrics has dimension $40$.
\end{example}

\section{Low degree curves}

We use the previous results to describe the spaces of curves with low degree. 

\subsection{Moduli spaces of curves}

First note that Theorem \ref{thm:Gamma-small-d} recovers a (weaker) special case of B\"arligea's result (see \cite{barligea:quasi-homogeneityI,barligea:quasi-homogeneityII}).

\begin{prop}
For $d \in \HH_2(X)$ with $d \leq d(\Theta)$, $M_{0,2}(X,d)$ has a dense $G$-orbit.
\end{prop}

\begin{proof}
By Theorem \ref{thm:Gamma-small-d}, the map $\ev : M_{0,2}(X,d) \to X^2$ is birational onto its image. By Theorem \ref{thm:curve-neigh-1}, the image is the $G$-schubert variety $G \times^P \Gamma_d(\{1.P\}) = G \times^P X_{z_d^P}$ which has a dense $G$-orbit.
\end{proof}

To give a precise description of the dense $G$-orbit $\Hilb_d^\circ(X)$ in $\Hilb_d(X)$, we compute the stabiliser $\Ssf = \Stab_G(C_{\gamma(d)})$ of the curve $C_{\gamma(d)}$ in $G$. We start with the following general result.

\begin{lemma}
\label{lemm:stab-inclusion}
Let $\gamma \in R$. Then $\Stab_G(C_\gamma)$ is contained in the subgroup of $G$ generated by $\SLT{\gamma}$ and $P \cap P^{s_\gamma}$.
\end{lemma}

\begin{proof}
Let $g \in \Stab_G(C_\gamma)$. Then up to left multiplying by an element of $\SLT{\gamma}$, we may assume that $g$ fixes $1.P$ and $s_\gamma.P$ and thus belongs to $P \cap P^{s_\gamma}$.
\end{proof}

Set $\Rsf = P \cap \Stab_G(z_d^P.P)$. Let $\Qsf \subset G$ be the parabolic subgroup containing $T$ and the root subgroups $U_\alpha$ for $\alpha$ such that $\scal{\alpha^\vee,\varpi - \frac{d}{2}\gamma(d)} \geq 0$. Let $L_{\Qsf} \subset \Qsf$ be the Levi subgroup of $\Qsf$ containing $T$.

\begin{prop}
\label{prop:stab-sd}
Assume that $d \leq d(\Theta)$. 
\begin{enumerate}
\item Then $\Ssf$ is generated by $\Rsf$ and $\SLT{\gamma(d)}$, in particular $\Ssf$ is connected.
\item The group $\Ssf$ is contained in the parabolic subgroup $\Qsf$.
\item $\Ssf \cap L_{\Qsf}$ is a connected symmetric subgoup of $L_{\Qsf}$.
\end{enumerate}
\end{prop}

\begin{proof}
For simplicity, set $\gamma = \gamma(d)$. Note that $\Rsf = P \cap P^{s_\gamma}$ by Corollary \ref{cor:z_d=gamma}.

(1) Recall that $C_{\gamma} = \SLT{\gamma}.P$. In particular $\SLT{\gamma} \subset \Ssf$. Theorem \ref{thm:Gamma-small-d} implies that $C_{\gamma}$ is the unique degree $d$ rational curve passing through $1.P$ and $z_d^P.P$. Thus it is stable under the stabiliser $\Rsf$ of this pair of points and $\Rsf \subset \Ssf$. The converse inclusion follows from Lemma \ref{lemm:stab-inclusion}. Since the intersection of two parabolic subgroups is connected (see \cite[Proposition 2.1]{digne.michel:representations}), the group $\Rsf$ is connected. The group $\Ssf$ is connected since it is generated by two connected groups.
%the group generated by $\SLT{\gamma(d)}$ and $\Rsf$ is connected.

%Conversely, let $U_\alpha \subset S_d$ and set $I = \{ \beta \in R^+ \ | \scal{\beta^\vee,s_{\gamma}(\varpi)} < 0\}$. We will need the facts that the maps $U_P^- \to B^-.P, u\mapsto u.P$ and $U_I \to Bs_{\gamma}.P, u \mapsto us_{\gamma}.P$ are isomorphisms. Note that $C_{\gamma} = s_{\gamma}.P \coprod U_{-\gamma}.P = U_{\gamma}s_{\gamma}. P \coprod 1.P$ and that we have $C_{\gamma} \cap B^-.P = U_{-\gamma}.P$ and $C_{\gamma} \cap Bs_\gamma.P = U_{\gamma}s_\gamma.P$. 

%First assume that $\alpha > 0$ with $\alpha \neq \gamma$. If $\alpha \in I$, then $U_\alpha U_\gamma \subset U_I$ and the map $U_\alpha U_\gamma \to Bs_\gamma .P$ is injective. This surface should however be contained in the curve $C_\gamma$, a contradiction. This implies that $U_\alpha \subset P^{s_\gamma}$ and thus $U_\gamma \subset R$.

%Assume now that $\alpha < 0$ with $\alpha \neq \gamma$. If $U_\alpha \subset U_P^-$, then $U_\alpha U_{-\gamma} \subset U_P^-$ and the map $U_\alpha U_{-\gamma} \to B^-.P$ is injective. This surface should however be contained in the curve $C_\gamma$, a contradiction. This implies that $U_\alpha \subset P$ and thus $U_\alpha \subset R$.

(2) We only need to check that both $\SLT{\gamma}$ and $\Rsf$ are contained in $\Qsf$. We have $\scal{\gamma^\vee,\varpi - \frac{d}{2}\gamma} = 0$ proving the inclusion $\SLT{\gamma} \subset \Qsf$. Let $\alpha$ be a root such that $U_\alpha \subset \Rsf$. The condition  $U_\alpha \subset P$ implies $\scal{\alpha^\vee,\varpi} \geq 0$, while the condition  $U_\alpha \subset \Stab(z_d^P.P) = \Stab(s_{\gamma}.P)= P^{s_{\gamma}}$ implies $\scal{\alpha^\vee,s_{\gamma}(\varpi)} \geq 0$ thus $\scal{\alpha^\vee,\varpi - d\gamma} \geq 0$. Taking the mean of these two inequalities, we get $\scal{\alpha^\vee,\varpi - \frac{d}{2}\gamma} \geq 0$ proving the claim.

(3) Let $x = \gamma^\vee(-1) \in G$ and define $\sigma : G \to G$ by $\sigma(g) = xgx^{-1}$. Since $x^2 = 1$, the interior automorphism $\sigma$ is an involution. Furthermore, since $x \in T \subset \Qsf$, the involution $\sigma$ preserves $\Qsf$ and $L_{\Qsf}$. For any root $\alpha$, let $u_\alpha : \bG_a \to U_\alpha$ be the group isomorphism identifying $U_\alpha$ to the additive one-dimensional group. Then $\sigma(u_\alpha(z)) = u_\alpha((-1)^{\scal{\gamma^\vee,\alpha}}z)$. The root space $U_\alpha$ is therefore contained in $G^\sigma$ if and only if $\scal{\gamma^\vee,\alpha}$ is even. Since $\gamma$ is a long root, this occurs if and only if $\alpha = \pm \gamma$ or $\scal{\gamma^\vee,\alpha} = 0$. If we assume that $U_\alpha \subset L_{\Qsf}$, which is equivalent to $2\scal{\alpha^\vee,\varpi} = d \scal{\alpha^\vee,\gamma}$, we have that $U_\alpha \subset G^\sigma$ if and only if $\alpha = \pm \gamma$ or $\scal{\alpha^\vee,\varpi} = 0 = \scal{\alpha^\vee,\gamma}$. This last condition is equivalent to $2\scal{\alpha^\vee,\varpi} = d \scal{\alpha^\vee,\gamma}$, $\scal{\alpha^\vee,\varpi} \geq 0$ and $\scal{\alpha^\vee,s_\gamma(\varpi)} \geq 0$ proving the claim. Let $U_{\Qsf}$ be the unipotent radical of $\Qsf$. Note that since $T \subset \Ssf$, we have $ \Ssf = ( \Ssf \cap L_{\Qsf}) \times ( \Ssf \cap U_{\Qsf})$ and $ \Ssf \cap U_{\Qsf}$ is a product of root subgroup therefore connected. Since $ \Ssf$ is connected, it follows that $ \Ssf \cap L_{\Qsf}$ is also connected.
\end{proof}

We now specialise to the cases $d = 1$ or $d = 2$. Recall that $U_{\Qsf}$ is the unipotent radical of $\Qsf$ and $L_{\Qsf}$ its Levi subgroup containing $T$. Recall also the definition of $\delta$.

\subsection{Lines} We first consider the case of lines \emph{i.e.} $d = 1 \leq d(\Theta)$.

\begin{prop}
Assume that $d = 1 \leq d(\Theta)$. 
\begin{enumerate}
\item If $\delta$ is long, then $\Ssf = \Qsf$.
\item If $G$ is not simply laced and $\delta$ is short, then $\Ssf \subsetneq \Qsf$.
\end{enumerate}
\end{prop}

\begin{remark}
 If $\delta$ is short, we will prove that we have the alternative:
\begin{enumerate}
\item $X \neq G_2/P_1$ in which case $U_{\Qsf} \subset \Ssf$ and $L_{\Qsf} \not\subset  S$,
\item $X = G_2/P_1$, in which case $L_{\Qsf} \subset S$ and $U_{\Qsf} \not\subset S$.
\end{enumerate}

\end{remark}

\begin{proof}
For simplicity, set $\gamma = \gamma(d)$. 

(1) Let $\alpha$ such that $U_\alpha \subset \Qsf$. Then $\scal{\alpha^\vee,\varpi} - \frac{1}{2} \scal{\alpha^\vee,\gamma} \geq 0$. 
If $\alpha$ is long, then $\scal{\alpha^\vee,\gamma} \in \{-1,0,1\}$. We get $\scal{\alpha^\vee,\varpi} = \scal{\alpha^\vee,\varpi} - \frac{1}{2}\scal{\alpha^\vee,\gamma} + \frac{1}{2}\scal{\alpha^\vee,\gamma} \geq -\frac{1}{2}$ and $\scal{\alpha^\vee,\varpi} - \scal{\alpha^\vee,\gamma} = \scal{\alpha^\vee,\varpi} - \frac{1}{2}\scal{\alpha^\vee,\gamma} - \frac{1}{2}\scal{\alpha^\vee,\gamma} \geq - \frac{1}{2}$. Since $\scal{\alpha^\vee,\varpi}$ and $\scal{\alpha^\vee,\varpi} - \scal{\alpha^\vee,\gamma}$ are integers, both are non-negative and $U_\alpha \subset \Ssf$. 

If $\alpha$ is short, set $a = (\delta,\delta)/(\alpha,\alpha) \in \{2,3\}$. Then $\scal{\alpha^\vee,\gamma} \in \{-a,0,a\}$ and the same argument shows that $\scal{\alpha^\vee,\varpi} \geq -\frac{a}{2}$ and $\scal{\alpha^\vee,\varpi} - \scal{\alpha^\vee,\gamma} \geq -\frac{a}{2}$. The value of $\scal{\alpha^\vee,\varpi}$ is the coefficient of $\delta^\vee$ in the expansion of $\alpha^\vee$ in terms of simple coroots. Since $\alpha^\vee$ is long and $\delta^\vee$ short this coefficient is a multiple of $a$. This implies that $\scal{\alpha^\vee,\varpi}$ and $\scal{\alpha^\vee,\varpi} - \scal{\alpha^\vee,\gamma}$ are non-negative and $U_\alpha \subset \Ssf$.

(2)  Let $\alpha$ such that $U_\alpha \subset U_{\Qsf}$. If $G$ is not of type $G_2$, then $|\scal{\alpha^\vee,\gamma}| \leq 2$ and we get $\scal{\alpha^\vee,\varpi} = \scal{\alpha^\vee,\varpi} - \frac{1}{2}\scal{\alpha^\vee,\gamma} + \frac{1}{2}\scal{\alpha^\vee,\gamma} \geq 0$ and $\scal{\alpha^\vee,\varpi} - \scal{\alpha^\vee,\gamma} = \scal{\alpha^\vee,\varpi} - \frac{1}{2}\scal{\alpha^\vee,\gamma} - \frac{1}{2}\scal{\alpha^\vee,\gamma} \geq 0$, thus $U_\alpha \subset \Ssf$. Furthermore, if $\alpha$ is the sum of all simple roots (in which case $\alpha$ is short), then $\scal{\alpha^\vee,\gamma} = 2$ and $\scal{\alpha^\vee,\varpi} = 1$, thus $\alpha$ is a root of $L_{\Qsf}$ but not a root of $S$. 
If $G = G_2$ and $\varpi = \varpi_1$ (with notations as in \cite{bourbaki:elements*78}), we have $\gamma = \Theta$ and an easy inspection shows that $L_\Qsf \subset \Ssf$ and that the roots $\alpha$ with ($U_\alpha \subset U_{\Qsf}$ and $U_\alpha \not\subset \Ssf$) are exactly  $2\alpha_1 + \alpha_2$ and  $-\alpha_1 - \alpha_2$.
\end{proof}

\begin{thm}
\label{thm:small-d=1}
$\Hilb_1(X)$ is $G$-spherical and homogeneous iff $\delta$ is long.
\end{thm}

\begin{proof}
If $\delta$ is long, then $\Ssf = \Qsf$ is a parabolic subgroup thus the dense orbit $G/\Ssf$ is compact and $\Hilb_1(X) = G/\Ssf$ is homogeneous. Thus we may assume that $\delta$ is short. If $X \neq G_2/P_1$, then $\Hilb_1(X)$ has an open subset of the form $G/\Ssf$ with $\Ssf \subset \Qsf$ and $U_{\Qsf} \subset \Ssf$. We thus have a map $G/\Ssf \to G/\Qsf$ with fiber $\Qsf/\Ssf \simeq L_{\Qsf}/(\Ssf \cap L_{\Qsf})$, the last isomorphism comes from the inclusion $U_{\Qsf} \subset \Ssf$. The quotient $L_{\Qsf}/(\Ssf \cap L_{\Qsf})$ is a symmetric space for $L_{\Qsf}$ and therefore spherical and affine. The quotient $G/\Ssf$ is therefore spherical by parabolic induction and not projective proving that $\Hilb_1(X)$ is not $G$-homogeneous. If $X = G_2/P_1$, we have a map $G/\Ssf \to G/\Qsf$ from the open orbit to a projective homogeneous space. The fiber is isomorphic to the product $U_{2\alpha_1 + \alpha_2} \times U_{-\alpha_1 - \alpha_2}$ (which is again affine). Since the two weights are linearly independent, the maximal torus has a dense orbit in this product, proving the result.
\end{proof}

For $d = 1$ and $\delta$ short, we describe the maps $G/\Ssf \to G/\Qsf$ and their fibers $\Qsf/\Ssf$.

\begin{example}
Let $X = \IG(p,2n)$ and $d = 1$. Then %In this case 
$\gamma = \Theta$ and $G/\Qsf = \IG(p-1,2n)$. We get a morphism $\Hilb_1^\circ(X) \to \IG(p-1,2n)$ defined by $f \mapsto \ker(f) := \bigcap_{x\in \bP^1} f(x)$ whose fiber over $V_{p-1}$ is the open $\Sp(V_{p-1}^\perp/V_{p-1})$-orbit in $\Gr(2,V^\perp_{p-1}/V_{p-1})$. This fiber is isomorphic to the symmetric space $\Sp_{2(n-p+1)}/(\Sp_2 \times \Sp_{2(n-p)})$. The complement is a hyperplane section which is isomorphic to $\IG(2,V^\perp_{p-1}/V_{p-1})$.
\end{example}

\begin{example}
Let $X = \OG(n,2n+1)$ and $d = 1$. Then %In this case 
$\gamma = \Theta$ and $G/\Qsf = \OG(n-2,2n+1)$. We get a morphism $\Hilb_1^\circ(X) \to \OG(n-2,2n+1)$ defined by $f \mapsto \ker(f) := \bigcap_{x\in \bP^1} f(x)$ whose fiber over $V_{n-2}$ is the open $\SO(V_{n-2}^\perp/V_{n-2})$-orbit of the Hilbert scheme of lines in $\OG(2,V^\perp_{n-2}/V_{n-2}) \simeq \bP^3$. This fiber is  isomorphic to the open $\SO_5$-orbit in $\bP^4$ thus isomorphic to the symmetric space $\SO_5/{\textrm{O}}_4$. The complement is a smooth quadric of dimension $3$.
%a $2$-to-$1$ cover of the open $\SO(V_{n-2}^\perp/V_{n-2})$-orbit in $\Gr(4,V^\perp_{n-2}/V_{n-2}) \simeq \bP^4$. 
%This fiber is isomorphic to the symmetric space $\SO_5/\SO_4$ and is an open subset of a smooth $4$-dimensional quadric.
\end{example}

\begin{example}
Let $X = F_4/P_4$ and $d = 1$. Then %In this case 
$\gamma = \Theta$ and we get a morphism $\Hilb_1^\circ(X) \to F_4/P_1$ whose fiber is the symmetric space $\Sp_6/(\Sp_4 \times \Sp_2)$ which is an open subset in $\Gr(2,6)$. The complement is a hyperplane section which is isomorphic to $\IG(2,6)$.
\end{example}

\begin{example}
Let $X = F_4/P_3$ and $d = 1$. Then %In this case 
$\gamma = \alpha_1 + \alpha_2 + 2\alpha_3 + 2\alpha_4$ with notation as in \cite{bourbaki:elements*78}. We get a morphism $\Hilb_1^\circ(X) \to F_4/(P_1 \cap P_4)$ whose fiber is the symmetric space $\Sp_4/(\Sp_2 \times \Sp_2)$ which is an open subset in $\Gr(2,4)$. The complement is a hyperplane section which is isomorphic to $\IG(2,4)$.
\end{example}

\begin{example}
Let $X = G_2/P_1$ and $d = 1$. Then %In this case 
$\gamma = \Theta$. We get a morphism $\Hilb_1^\circ(X) \to G_2/P_1$ whose fiber is $\bG_a^2$.
\end{example}

\subsection{Conics} We now turn to the case $d = 2 \leq d(\Theta)$.

\begin{prop}
If $d = 2 \leq d(\Theta)$, then $U_{\Qsf} \subset \Ssf$ except if $X = F_4/P_3$.
\end{prop}

\begin{proof}
Set $\gamma = \gamma(d)$. Let $\alpha \in R$ such that $U_\alpha \subset U_{\Qsf}$ \emph{i.e.} $\scal{\alpha^\vee,\varpi} > \frac{d}{2}\scal{\alpha^\vee,\gamma}$. Note also that this implies that $\alpha \neq \pm \gamma$. If $\alpha$ is long (for example if $G$ is simply laced), then $|\scal{\alpha^\vee,\gamma}| \leq 1$ and we get $\scal{\alpha^\vee,\varpi} \geq 0$ and  $\scal{\alpha^\vee,\varpi - 2\gamma} \geq 0$. Therefore $U_\alpha \subset \Ssf$. If $\scal{\alpha^\vee,\gamma} = 0$, then $\scal{\alpha^\vee,\varpi} = \scal{\alpha^\vee,\varpi -2\gamma} = \scal{\alpha^\vee,\varpi - \gamma} > 0$ thus $U_\alpha \subset \Ssf$. We may therefore assume that $\alpha$ is short such that $\scal{\alpha^\vee,\gamma} \neq 0$.

We now proceed case by case. All simply laced cases were treated by the previous argument. Note that the condition $d(\Theta) \geq 2$ implies that $G$ is not of type $C_n$. 

Assume that $G$ is of type $G_2$, then $\varpi = \varpi_2$ and $X = G_2/P_2$. But in this case $\gamma = \Theta = \varpi$ and $\Qsf = G$. The unipotent subgroup $U_\Qsf$ is trivial proving the result.

Assume that $G$ is not of type $G_2$. Since $\alpha$ is short, $\gamma$ long and $\scal{\alpha^\vee,\gamma} \neq 0$, we have $\scal{\alpha^\vee,\gamma} = \pm 2$. Assume that $\scal{\alpha^\vee,\varpi - \gamma} \geq 2$, then $\scal{\alpha^\vee,\varpi} \geq 2 + \scal{\alpha^\vee,\gamma} \geq 0$ and $\scal{\alpha^\vee,\varpi - 2\gamma} \geq 2 - \scal{\alpha^\vee,\gamma} \geq 0$, therefore $U_\alpha \subset \Ssf$. Assume that $|\scal{\beta^\vee,\varpi}| \leq 2$ for any root $\beta$. If $\scal{\alpha^\vee,\gamma} = 2$, then $\scal{\alpha^\vee,\varpi} > \scal{\alpha^\vee,\gamma} = 2$. This is impossible. If $\scal{\alpha^\vee,\gamma} = -2$, then $\scal{s_{\gamma}(\alpha^\vee),\varpi} = \scal{\alpha^\vee,\varpi} - 2\scal{\alpha^\vee,\gamma} > - \scal{\alpha^\vee,\gamma} = 2$. This is impossible.

We may therefore assume that $G$ is not simply laced and that there exists a root $\beta$ with $\scal{\beta^\vee,\varpi} > 2$. This only occurs if $G$ is of type $F_4$ and $\varpi = \varpi_2$ or $\varpi_3$. Furthermore, we may assume that $\scal{\alpha^\vee,\varpi - \gamma} = 1$ and $|\scal{\alpha^\vee,\gamma}| = 2$. This implies that $\scal{\alpha^\vee,\varpi}$ is odd. Since $\alpha$ is short this is only possible if $\varpi = \varpi_3$.

Note that if $X = F_4/P_3$, the group $U_{\Qsf}/(\Ssf \cap U_{\Qsf})$ has dimension $4$ and is generated by the groups $U_\alpha$ for $\alpha$ in the following list of roots (with notation as in \cite{bourbaki:elements*78}): $\alpha_1 + 2\alpha_2 + 3\alpha_3 + 2\alpha_4$, $\alpha_1 + 2\alpha_2 + 3\alpha_3 + \alpha_4$, $-(\alpha_1 + \alpha_2 + \alpha_3)$, $-(\alpha_1 + \alpha_2 + \alpha_3 +\alpha_4)$.
\end{proof}

\begin{thm}
\label{thm:small-d}
If $d(\Theta) \geq 2$, then $\Hilb_2(X)$ is $G$-spherical except if $X = F_4/P_3$.
\end{thm}

\begin{proof}
The fact that $\Hilb_2(F_4/P_3)$ is not spherical was proved in Example \ref{exam:F4/P3}. In all other cases, $\Hilb_2(X)$ has an open subset of the form $G/\Ssf$ with $\Ssf \subset \Qsf$ and $U_{\Qsf} \subset \Ssf$. We thus have a map $G/\Ssf \to G/\Qsf$ with fiber $\Qsf/\Ssf \simeq L_{\Qsf}/(\Ssf \cap L_{\Qsf})$, the last isomorphism comes from the inclusion $U_{\Qsf} \subset \Ssf$. The quotient $L_{\Qsf}/(\Ssf \cap L_{\Qsf})$ is a symmetric space for $L_{\Qsf}$ and therefore spherical. The quotient $G/\Ssf$ is therefore spherical by parabolic induction.
\end{proof}

For $d = 2$, we describe some of the maps $G/\Ssf \to G/\Qsf$ and their fibers $\Qsf/\Ssf$.

\begin{example}
\label{ex:E8P1}
Let $X = E_8/P_1$ and $d = 2$. In this case $\gamma = \Theta$. We get a morphism $\Hilb_2^\circ(X) \to E_8/P_8$ whose fiber is the symmetric space $E_7/(D_6 \times \SL_2)$. Geometrically, a general conic is contained in a unique subvariety $Y$ of $X$ isomorphic to $E_7/P_1$ and varying $Y$ in its family (which is isomorphic to $E_8/P_8)$, we get that $\Hilb_2^\circ(X)$ is covered by the Hilbert schemes $\Hilb_2^\circ(Y) \simeq E_7/(D_6 \times \SL_2)$. 
\end{example}

\begin{example}
Let $X = E_6/P_4$ and $d = 2$. In this case $\gamma = \alpha_1 + \alpha_2 + 2\alpha_3 + 2\alpha_4 + 2 \alpha_5 + \alpha_6$ with notation as in \cite{bourbaki:elements*78}. We get a morphism $\Hilb_2^\circ(X) \to E_6/(P_1 \cap P_6)$ whose fiber is the symmetric space $\SO_8/(\SO_4 \times \SO_4)$.
%{\rm S}(\rm{O}(4) \times {\rm O}(4))$. 
The geometric interpretation of Example \ref{ex:E8P1} adapts easily in this case.
\end{example}

%%%%%%%%%%%%%%%%%%%%%%%%%%%%%%%%%%%%%%%%%%%%%%%%%%%%%%
%\input{geometry}

\section{Geometric method}
\label{section:geometric}

We present a geometric method to deal with the case $d(\Theta) < 2 \leq d_X$. It covers all cases not covered by Theorem \ref{thm:small-d} and works when $G$ is classical or when $X$ is a hermitian symmetric space (\emph{i.e.} cominuscule). For a smooth quadric $Q$, let $\SO(Q)$ and ${\rm Spin}(Q)$ be the special orthogonal and the spinor groups acting on $Q$.

\subsection{Projective spaces and quadrics} We first prove the following easy results. 

\begin{prop}
\label{prop:proj-space}
The Hilbert scheme $\Hilb_2(\bP^n)$ is $\SL_{n+1}$-spherical.
\end{prop}

\begin{proof}
A general conic determines a unique plane. We therefore have a rational map $\Hilb_2(X) \dasharrow \Gr(3,n+1)$ whose fiber over $V_3$ is the Hilbert scheme $\Hilb_2(\bP^2)$ of conics on $\bP(V_3)$. The group $\SL(V_3)$ is contained in $\Stab_{\SL_{n+1}}(V_3)$ and acts on $\Hilb_2(\bP^2)$. We are left with the case $n = 2$ for which the result is classical: $\Hilb_2(\bP^2)$ is birational to the symmetric space $\SL_3/\textrm{Z}(\SL_3) \times \SO_3 \simeq \PGL_3/\PSO_3$.
\end{proof}
 
 \begin{prop}
\label{prop:conic-sp}
 The Hilbert scheme $\Hilb_2(\bP^{2n-1})$ is not $\Sp_{2n}$-spherical.
 \end{prop}

\begin{proof}
Let $C_2$ be a general conic in $\bP^{2n-1}$. Then $C_2$ is contained in a unique plane $\Pi$ and we have a rational map $\Hilb_2(X) \dasharrow \Gr(3,2n)$ that maps a conic to its span. The fiber over $\Pi$ is given by $\Hilb_2(\Pi)$ acted on by $H = \Stab_{\Sp_{2n}}(\Pi)$. Let $\ell$ be the kernel of the symplectic form on $\Pi$. Then $\dim\ell = 1$ and $H \subset \Qsf$ with $\Qsf = \Stab_{\Sp_{2n}}(\ell)$. Note that $\Qsf$ is a parabolic subgroup. Let $L_\Qsf$ be the levi subgroup of $\Qsf$ and $U_\Qsf$ be its unipotent radical. Then $\Qsf \simeq L_\Qsf \ltimes U_\Qsf$ with $L_\Qsf \simeq \GL_1 \times \Sp_{2n-2}$ and $H \simeq (\GL_1 \times \Sp_2 \times \Sp_{2n-4}) \ltimes U_\Qsf$. Let $B'$ be a Borel subgroup opposite to $U_\Qsf$ and such that $(B' \cap L_\Qsf).(\Sp_2 \times \Sp_{2n-4})$ is dense in $L_\Qsf/\Sp_2 \times \Sp_{2n-4}$. Then $B' \cap \Sp_{2n-4}$ acts trivially on $\Hilb_2(\Pi)$. Thus $B'$ acts on $\Hilb_2(\Pi)$ via $B' \cap (\GL_1 \times \Sp_2)$ which has dimension $2$, while $\Hilb_2(\Pi)$ has dimension $5$, proving the result.
\end{proof}

Let $X = Q_n$ be a smooth quadric of dimension $n$.% and let $G$ be the special orthogonal group acting on $X$.

\begin{prop}
\label{prop:quadrics}
The Hilbert scheme $\Hilb_2(Q_n)$ is $\SO_{n+2}$-spherical.
\end{prop}

\begin{proof}
A general conic is defined by a plane in the minimal embedding of $X$. Thus $\Hilb_2(X)$ is birational to $\Gr(3,n+2)$ which is $G$-birational to the symmetric space $\SO_{n+2}/{\rm S}({\rm O}_3 \times {\rm O}_{n-1})$.
\end{proof}

\begin{remark}
\label{remark:connected}
Note that for $X = Q_n$, the stabiliser of a general conic $C_2$ in $\SO_{n+2}$ is $\Stab_{\SO_{n+2}}(C_2) = {\rm S}({\rm O}_3 \times {\rm O}_{n-1})$ which is never connected. Replacing $\SO_{n+2}$ with $\PSO_{n+2}$ it is easy to check that $\Stab_{\SO_{n+2}}(C_2)$ is connected if and only if $n$ is even.
\end{remark}

\begin{remark}
For $X$ a projective space or a quadric, we have $\Gamma_2 = X$.
\end{remark}

\subsection{Quantum to classical principle}

This construction was first introduced by Buch in \cite{buch:quantum} and developed by Buch, Kresch and Tamvakis for classical groups (see   \cite{buch.kresch.ea:gromov-witten,buch.kresch.ea:quantum} and references therein). It was later generalised by Chaput, Manivel and the second author in a series of papers \cite{chaput.manivel.ea:quantum*1,chaput.manivel.ea:quantum,chaput.manivel.ea:affine,chaput.manivel.ea:quantum*2}. We start with the following definition.

\begin{defn} The homogeneous space $X$ is \emph{cominuscule} if the coefficient of $\delta$ in the expansion of $\Theta$ in terms of simple roots is equal to $1$.
\end{defn}

For $X$ cominuscule, we have $d(\Theta) = 1$. We summarise some results of \cite{chaput.manivel.ea:quantum*1}. 

\begin{prop}
\label{prop:geom-comin}
Assume that $X$ is cominuscule different from a projective space.
\begin{enumerate}
\item $\Gamma_2$ is a smooth quadric $Q$ given in Table \ref{table}.
\item $\Stab_G(\Gamma_2)$ contains $\SO(\Gamma_2)$.
%\item $\Gamma_2(x,y)$ is a translate of $\Gamma_2$.
\item A general conic is contained in a unique translate of $\Gamma_2$.
\item The stabiliser of $\Gamma_2$ is a parabolic subgroup of $G$  containing $\SO(Q)$.
\end{enumerate}
\begin{table}[ht]
\begin{tabular}{c|c|c}
Type of $G$ & $X$ & $\Gamma_2$ \\
\hline
$A_n$ & $\Gr(p,n+1)$ & $\Gr(2,4) \simeq Q_4$ \\
$B_n$ & $Q_{2n-1}$ & $Q_{2n-1}$ \\
$C_n$ & $\LG(n,2n)$ & $\LG(2,4) \simeq Q_3$ \\
$D_n$ & $Q_{2n-2}$ & $Q_{2n-2}$ \\
$D_n$ & $\OG(n,2n)$ & $\OG(4,8) \simeq Q_6$ \\
$E_6$ & $E_6/P_6$ & $Q_8$\\
$E_7$ & $E_7/P_7$ & $Q_{10}$. \\
\hline
\end{tabular}
~\\
\centering
~\\
\caption{\label{table} Double curve neighborhoods $\Gamma_2$ for cominuscule spaces $X$.}
\end{table}
%\FloatBarrier
\end{prop}

\begin{proof}
See \cite[Proposition 18 and Proposition 19]{chaput.manivel.ea:quantum*1}. 
\end{proof}

\begin{cor}
\label{cor:geometric-comin}
If $X$ is cominuscule, then $\Hilb_2(X)$ is $G$-spherical.
\end{cor}

\begin{proof}
We may assume that $X$ is not a projective space (Proposition \ref{prop:proj-space}).
Let $C_2$ be a general conic passing through $1.P$ and $z_2^P.P$. Then $C_2 \subset \Gamma_2$ and $G.C_2$ is an open orbit in $\Hilb_2(X)$. Proposition \ref{prop:geom-comin} implies that we have a map $G.C_2 \to G.\Gamma_2^X = G/\Stab_G(\Gamma_2)$ and the quotient $G/\Stab_G(\Gamma_2)$ is projective. The fiber of this map is the Hilbert scheme $\Hilb_2(\Gamma_2)$ of conics in $\Gamma_2$. Since $\Gamma_2$ is a smooth quadric $Q$ and since $\Stab_G(\Gamma_2)$ contains $\SO(Q)$, the result follows from Proposition \ref{prop:quadrics}.
\end{proof}

If $d(\Theta) < 2$ and $X$ is not cominuscule, then $G$ is non-simply laced. Table \ref{table-non-comin} gives the list of possible $X$. We deal with each case separately. Note that Example \ref{exam:F4/P3} and Proposition \ref{prop:conic-sp} prove that $\Hilb_2(X)$ is not $G$-spherical for $X = \SO_{7}/P_{3}$, $\Sp_{2n}/P_{1}$, $F_4/P_4$ and $G_2/P_1$. 

\begin{table}[ht]
\begin{tabular}{c|c}
Type of $G$ & $X$  \\
\hline
$B_n$ & $\OG(n,2n+1)$ \\
$C_n$ & $\IG(p,2n)$, $p < n$. \\
$F_4$ & $F_4/P_4$ \\
$G_2$ & $G_2/P_1$ \\
\hline
\end{tabular}
~\\
\centering
~\\
\caption{\label{table-non-comin} Non-cominuscule varieties $X$ with $d(\Theta) = 1$.}
\end{table}

\begin{prop}
\label{prop:geom-B}
Assume that $X = \OG(n,2n+1)$ with $G = \SO_{2n+1}$ for $n \ge 4$.
\begin{enumerate}
\item $\Gamma_2$ is a smooth quadric of dimension $6$.
\item A general conic is contained in a unique translate of $\Gamma_2$.
\item $\Stab_G(\Gamma_2)$ contains $\SO(\Gamma_2)$.
\item We have a $G$-map $G.\Gamma_2 \to \OG(n-4,2n+1)$ with fiber $\Gr(8,9)$.
\item $\Hilb_2(X)$ is not $G$-spherical.
\item $\Hilb_2(X)$ is $\Aut(X)$-spherical.
\end{enumerate}
\end{prop}

\begin{proof}
It is a classical result that $\Aut(X)$ is of type $D_{n+1}$ and that $X$ is isomorphic to $\OG(n+1,2n+2)$. This remark and Corollary \ref{cor:geometric-comin} prove (6). Proposition \ref{prop:geom-comin} implies that (1) and (2) hold: $\Gamma_2 = \OG(4,8)$ is a smooth quadric of dimension~$6$. Let $E$ be a vector space of dimension $2n+1$ endowed with a non-degenerate quadratic form. $\Gamma_2$ is determined by an isotropic subspace $E_{n-4} \subset E$ of dimension $n-4$ together with a (general) subspace $E_{n+4} \subset E$ of dimension $n+4$ such that $E_{n-4} \subset E_{n+4} \subset E_{n-4}^\perp$. Then $\Gamma_2 = \Gamma_2(x,y)$ is the connected component containing $x$ (and $y$) of 
$\Gamma_\circ = \{ V_n \in X \ | \ E_{n-4} \subset V_n \subset E_{n+4} \}$
and is isomorphic to $\OG(4,E_{n+4}/E_{n-4}) \simeq \OG(4,8)$.
The map $G.\Gamma_2 \to \OG(n-4,2n+1), g.\Gamma_2 \mapsto g.E_{n-4}$ satisfies assertion~(4). We are left with assertions (3) and (5). Let $C_2$ be a general conic passing through $x = 1.P$ and $y = z_d^P.P$. We have a map $G.C_2 \to G.\Gamma_2$ whose fiber is an open subset of $\Hilb_2(\Gamma_2) \simeq \Hilb_2(Q_6)$. Let $\Qsf = \Stab_G(E_{n-4})$. The Levi subgroup of $\Qsf$ is isomorphic to $\GL_{n-4}\times \SO_9$ proving (3). Furthermore, the $\GL_{n-4}$ part of this Levi subgroup acts trivially on the fiber $F$ of the map $G.C_2 \to \OG(n-4,2n+1)$. Therefore it is enough to prove that $F$ is not $\SO_9$-spherical. But $\dim F = \dim \Hilb_2(X) - \dim \OG(n-4,2n+1) = \dim \Gr(8,9) + \dim \Hilb_2(\Gamma_2) = 23$ while the dimension of a Borel subgroup in $\SO_9$ is $20$, proving the claim.
\end{proof}

\begin{prop}
\label{prop:geom-C}
Assume that $X = \IG(p,2n)$ with $G = \Sp_{2n}$ and $1<p < n$.
\begin{enumerate}
\item $\Gamma_2$ is a smooth quadric of dimension $3$.
\item A general conic is contained in a unique translate of $\Gamma_2$.
\item $\Stab_G(\Gamma_2)$ contains $\SO(\Gamma_2)$.
\item We have a $G$-map $G.\Gamma_2 \to \IG(p-2,2n)$ with fiber $\Gr(4,2(n+2-p))$.
\item $\Hilb_2(X)$ is not $G$-spherical.
\end{enumerate}
\end{prop}

\begin{proof}
Let $E$ be a symplectic vector space of dimension $2n$. We view $\IG(p,2n)$ as a subvariety of $\Gr(p,E)$. By Proposition \ref{prop:geom-comin}, a general conic in $X$ is contained in a unique subvariety $\{ V_p \in X \ | \ E_{p-2} \subset V_p \subset E_{p+2} \}$ where $E_{p-2}$ and $E_{p+2}$ are subspaces of $E$ of dimension $p - 2$ and $p + 2$. Furthermore, since all subspaces of the conic are isotropic, we get that $E_{p-2} \in \IG(p-2,2n)$ while $E_{p+2}/E_{p-2} \in \Gr(4,E_{p-2}^\perp/E_{p-2})$. In that case $\Gamma_2 = \Gamma_2(x,y)$ is a translate  of the following set
$$\Gamma_\circ = \{ V_p \in X \ | \ E_{p-2} \subset V_p \subset E_{p+2} \} \simeq \LG(2,E_{p+2}/E_{p-2}) \simeq \LG(2,4).$$
This proves (1) and (2). The map $G.\Gamma_2 = G.\Gamma_\circ \to \IG(p-2,2n), g.\Gamma_\circ \mapsto g.E_{p-2}$ satisfies assertion (4). 
Let $C_2$ be a general conic passing through $x = 1.P$ and $y = z_d^P.P$. Then $\Gamma_2$ is the only translate of $\Gamma_\circ$ containing $C_2$ and we have a map $G.C_2 \to G.\Gamma_2$ whose fiber is an open subset of $\Hilb_2(\Gamma_2) \simeq \Hilb_2(Q_3)$. Let $\Qsf = \Stab_G(E_{p-2})$. The Levi subgroup of $\Qsf$ is isomorphic to $\GL_{p-2}\times \Sp_{2(n+2-p)}$. Furthermore, the $\GL_{p-2}$ part of this Levi subgroup acts trivially on the fiber $F$ of the map $G.C_2 \to \IG(p-2,2n)$. Therefore it is enough to prove that $F$ is not $\Sp_{2(n+2-p)}$-spherical. We have a map $F \to \Gr(4,2(n+2-p))$ with fiber an open subset of $\Hilb_2(\Gamma_2)$. The variety $\Gr(4,2(n+2-p))$ is a symmetric space for $\Sp_{2(n+2-p)}$ whose open orbit is $\Sp_{2(n+2-p)}/\Sp_4 \times \Sp_{2(n-p)}$. The group $\Sp_{2(n-p)}$ acts trivially on $\Hilb_2(\Gamma_2)$ and $\SO_5 \simeq \Sp_4$ acts on $\Gamma_2$, proving (3). Let $B'$ be a Borel subgroup such that $B'.(\Sp_4 \times \Sp_{2(n-p)})$ is dense in $\Sp_{2(n+2-p)}/\Sp_4 \times \Sp_{2(n-p)}$, then $\dim B' \cap (\Sp_4 \times \Sp_{2(n-p)}) = 8(n-p)$. Assume that $2(n-p) > 4$. The Satake diagram of this symmetric space proves that $B'$ contains the Borel subgroup of a subgroup $\Sp_{2(n-p-2)} \subset \Sp_{2(n-p)}$. In particular, the projection of $B' \cap (\Sp_4 \times \Sp_{2(n-p)})$ to $\Sp_4$ has dimension at most $4$. If $2(n-p) \leq 4$, then $\dim B' \cap (\Sp_4 \times \Sp_{2(n-p)}) \leq 4$. In any case, the projection of $B' \cap (\Sp_4 \times \Sp_{2(n-p)})$ to $\Sp_4$ has dimension at most $4$ while $\Hilb_2(\Gamma_2)$ has dimension $6$ and therefore $B'$ has no dense orbits in $\Hilb_2(X)$.
\end{proof}

A geometric construction also exists for $X = F_4/P_4$.

\begin{prop}
\label{prop:geom-F}
Assume that $X = F_4/P_4$ and $G = F_4$. 
\begin{enumerate}
\item Then $\Gamma_2$ is a smooth quadric of dimension $7$.
\item A general conic is contained in a unique translate of $\Gamma_2$.
\item $\Stab_G(\Gamma_2) = {\rm Spin}(\Gamma_2)$.
\item $\Hilb_2(X)$ is not $G$-spherical.
\end{enumerate}
\end{prop}

\begin{proof}
(1) We use the fact that $F_4/P_4$ can be realised as a general hyperplane section $H$ of $E_6/P_6$ (see for example \cite{benedetti.perrin:cohomology}). A general conic in $E_6/P_6$ is contained in a unique translate of $\Gamma_2^{E_6/P_6}$ which is a smooth quadric of dimension $8$. This proves that a general conic in $X$ is contained in a translate of a general hyperplane section of $\Gamma_2^{E_6/P_6}$ which is a smooth quadric of dimension $7$. Furthermore, $\Gamma_2^X \subset \Gamma_2^{E_6/P_6} \cap H$ and since $\Gamma_2^{E_6/P_6} \cap H$ is a smooth quadric, any point in $\Gamma_2^{E_6/P_6} \cap H$ is contained on a conic passing through $1.P$ and $z_d^P.P$ thus $\Gamma_2^X = \Gamma_2^{E_6/P_6} \cap H$.

Let $C_2$ be a general conic passing through  $x = 1.P$ and $y = z_d^P.P$. We have a map $G.C_2 \to G.\Gamma_2^X$. Since $E_6.\Gamma_2^{E_6/P_6} \simeq E_6/P_1$, we get that $G.\Gamma_2^X$ is the open $G$-orbit in $E_6/P_1$. This orbit is isomorphic to $F_4/{\rm Spin_9}$. We thus have $\Stab_G(\Gamma_2^X) = {\rm Spin_9}$, proving (3), and an equivariant map $G.C_2 \to F_4/{\rm Spin_9}$ with fiber given by an open subset of $\Hilb_2(\Gamma_2^X) = \Hilb_2(Q_7)$. If $B'$ is a Borel subgroup of $G$ such that $B'.{\rm Spin}_9$ is dense in $F_4/{\rm Spin}_9$, then $\dim(B' \cap {\rm Spin}_9) = 12$ while $\dim\Hilb_2(Q_7) = 18$ proving that $\Hilb_2(X)$ is not $G$-spherical.
\end{proof}

\begin{remark}
\label{remark:G2}
If $X = G_2/P_1$, then $\Gamma_2 = X$.
\end{remark}

We give a result on conics on the $G_2$-quadric $X = G_2/P_1$ which is probably well known to experts. We thank Laurent Manivel for provinding the following proof.

\begin{prop}
The group $G_2$ has no dense orbit on $\Hilb_2(X)$. More precisely, the complexity of the action is $1$.
\end{prop}

\begin{proof}
We prove that a general $G_2$-orbit has codimension $1$ in $\Hilb_2(X)$. Recall some notations, we refer to \cite[Subsection 2.1.2]{manivel:cayley} for more details. Let $V_7$ be the unique irreducible $7$-dimensional $G_2$-representation. If $T$ is a maximal torus, its weights are given by short roots and $0$. Fix some set of positive roots and let $(\alpha_i)_{i \in [1,3]}$ be the positive short roots. For all $i \in [1,3]$, choose $T$-eigenvectors $v_{\pm i} \in V_7$ of weight $\pm\alpha_i$ and $v_0$ of weight $0$. We may choose these vectors so that the $G_2$-invariant bilinear form $q$ (we will also denote by $q$ the associated bilinear form) is given by 
$$q\left(\sum_{i \in [1,3]}(\lambda_i v_i + \lambda_{-i}v_{-i}) + \lambda_0v_0\right) = \sum_{i = 0}^3 \lambda_i\lambda_{-i}$$ 
and the $G_2$-invariant trilinear form $\Omega$ is given by
$$\Omega = v_0 \wedge \left( \sum_{i = 1}^3 v_i \wedge v_{-i} \right) + v_1 \wedge v_2 \wedge v_3 +  v_{-1} \wedge v_{-2} \wedge v_{-3}.$$
 
Since $X = Q_5 \subset \bP(V_7)$ is a $5$-dimensional quadric, the Hilbert scheme $\Hilb_2(X)$ is birational to $Y = \Gr(3,V_7)$. We prove that a general $G_2$-orbit in $Y$ has codimension $1$. If $\Pi \in \Gr(3,V_7)$ is a general $3$-dimensional subspace in $V_7$, then the $G_2$-invariant trilinear form $\Omega$ has a non-trivial $1$-dimensional kernel $K \subset \Pi^\perp$ (since $\dim(\Pi^\perp) = 4$). Note that $\Stab_{G_2}(\Pi) \subset \Stab_{G_2}(K)$. Since $\Pi$ is general, this kernel is generated by a non-isotropic vector and up to $G_2$-action, we may assume that $K = \scal{v_0}$. We now want to understand the subspaces $\Pi \subset v_0^\perp \subset V_7$ for which $v_0 \in \Ker(\Omega\vert_{\Pi^\perp})$ and compute $\Stab_{G_2}(\Pi) \subset \Stab_{G_2}(v_0)$.

The stabiliser $S$ of $v_0$ is, up to a finite group, isomorphic to $\SL_3$. Explicitly $S$ is the subgroup associated to the subsystem of long roots. The representation $V_7$ decomposes as $S$-representation as $V_7 = \scal{v_0} \oplus V_3 \oplus V_3^\vee$ where $v_0^\perp = V_3 \oplus V_3^\vee$, $V_3 =  \scal{v_3,v_{-1},v_{-2}}$ and $V_3^\vee = \scal{v_{-3},v_1,v_2}$ (here we assume that $\alpha_3 = \alpha_1 + \alpha_2$ is the highest short root). Note that both $V_3$ and $V_3^\vee$ are isotropic for $q$. We identify $3$-dimensional subspaces $\Pi \subset v_0^\perp = V_3 \oplus V_3^\vee$ not meeting $V_3$ to linear maps $f : V_3 \to V_3^\vee$ via $f \mapsto \Pi_f := \scal{v + f(v) \ | \ v \in V_3}$. If $f^\vee$ is the adjoint of $f$ with respect to $q$, then $\Pi_f^\perp = \scal{v - f^\vee(v) \ | \ v \in V_3}$. The subspace $\Pi_f$ associated to $f$ satisfies $v_0 \in \Ker(\Omega\vert_{\Pi_f^\perp})$ if the $2$-form $\omega = \Omega(v_0,-,-)$ vanishes on $\Pi_f^\perp$. We have
$$\omega = \sum_{i = 1}^3 v_i \wedge v_{-i},$$
thus $\omega(V_3,V_3) = 0 = \omega(V_3^\vee,V_3^\vee)$ and $\omega(v,\varphi) = q(v,\varphi)$ for $v \in V_3$ and $\varphi \in V_3^\vee$. The condition $\omega(v - f^\vee(v),w - f^\vee(w)) = 0$ therefore translates to $q(v,f^\vee(w)) = q(f^\vee(v),w) = q(v,f(w))$ for all $v,w \in V_3$. This is equivalent to $f^\vee = f$. The subspaces $\Pi_f$ satisfying $v_0 \in \Ker(\Omega_{\Pi_f^\perp})$ are thus exactly given by symmetric maps $f : V_3 \to V_3^\vee$. The stabiliser in $\SL_3$ of a general element of this form is the stabiliser of $q : V_3 \to V_3^\vee$ and is thus isomorphic to $\SO(q) \cap \SL_3 = \SO_3$. The general stabiliser is therefore of dimension $3$ and the general $G_2$-orbit is of dimension $\dim(G_2) - \dim(\SO_3) = 11$ which has codimension $1$ in $\Gr(3,V_7)$ proving the claim.
\end{proof}

The following statement is a consequence of Theorem \ref{thm:Gamma-small-d}, Propositions \ref{prop:geom-comin}, \ref{prop:geom-B}, \ref{prop:geom-C}, \ref{prop:geom-F} and Remark \ref{remark:G2}.
%is the following statement.

\begin{prop}
Assume that $X$ is not a projective space.
\begin{enumerate}
\item A general conic is contained in a unique translate of $\Gamma_2$.
\item $\Gamma_2$ is a smooth quadric.
\item $\Stab_G(\Gamma_2)$ is always connected and spherical except for $X = F_4/P_3$.
\end{enumerate}
\end{prop}

\begin{proof}
We only need to prove the connectedness of $\Stab_G(\Gamma_2)$. For $X = G_2/P_1$, we have $\Gamma_2 = X$ and $\Stab_G(\Gamma_2) = G$ is connected. In the other cases, we claim that there is a surjective group morphism ${\rm Spin}(\Gamma_2) \times (P_x \cap P_y) \to \Stab_G(\Gamma_2)$. This proves the connectedness of $\Stab_G(\Gamma_2)$ since ${\rm Spin}(\Gamma_2)$ and $P_x \cap P_y$ are connected (for the later apply  \cite[Proposition 2.1]{digne.michel:representations}).

We now prove the claim. We have a group morphism $\varphi : {\rm Spin}(\Gamma_2) \to \Stab_G(\Gamma_2)$ with finite kernel and an inclusion $P_x \cap P_y \subset \Stab_G(\Gamma_2(x,y)) = \Stab_G(\Gamma_2)$. This gives a group morphism ${\rm Spin}(\Gamma_2) \times (P_x \cap P_y) \to \Stab_G(\Gamma_2)$, $(s,h) \mapsto \varphi(s)h$. Now for $g \in \Stab_G(\Gamma_2)$, both $g.x$ and $g.y$ lie in $\Gamma_2$. Since ${\rm Spin}(\Gamma_2) \to \Aut(\Gamma_2)^\circ$ is surjective, there exists $s \in {\rm Spin}(\Gamma_2)$ with $s.x = g.x$ and $s.y = g.y$. In particular $h = \varphi(s)^{-1}g \in P_x \cap P_y$ and $g = \varphi(s)h$ proving the claim.
\end{proof}

%%%%%%%%%%%%%%%%%%%%%%%%%%%%%%%%%%%%%%%%%%%%%%%%%%%%%%
%\input{local-structure-hilb}

\section{Local structure of Hilbert schemes of conics} \label{section:local}
We now present the second proof of Theorem \ref{thm:main}.
It consists of three steps: 
\begin{enumerate}
\item Show that $\Hilb_{2}(X)$ is smooth (see Corollary \ref{cor:hilb-is-smooth}),  
\item compute the normal space of $\Hilb_{2}(X)$ along a closed orbit (see Section \ref{section:representation}), 
\item  apply the local structure theorem for spherical varieties (see Lemma \ref{lemma:loc-str-spherical}).
\end{enumerate}
We keep notations of Section \ref{section:notations}, but the arguments are independent of Section \ref{section:curves}-\ref{section:geometric}.

\subsection{Local structure of spherical varieties}

Let $H$ be a connected reductive group and choose $T_{H} \subset B_H \subset H$ a maximal torus and a Borel subgroup of $H$. Recall that an $H$-variety is spherical if it is normal and contains a dense $B_H$-orbit. Denote by $\mathfrak{X}(T_{H}) := \Hom_{\Z}(T_{H},\, \bG_m) (\simeq \Hom_{\Z}(B_{H},\, \bG_m))$ the character lattice.

\begin{defn}
Let $Z$ be $H$-spherical.
\begin{enumerate}
        \item The \emph{weight lattice $\Lambda_{H}(Z)$ of $Z$} is the sublattice of $\mathfrak{X}(T_{H})$ defined by
        \[
         \Lambda_{H}(Z) = \left\{\chi_{f} \in \mathfrak{X}(T_{H}) \ | \  b.f = \chi_{f}(b) f,\, \forall b \in B_{H} \text{ for some } 0\not= f \in \C(Z)^{(B_{H})} \right\}.
        \]
%The \emph{weight lattice $\Lambda_{H}(Z)$ of $Z$} is a sublattice of the character lattice $\mathfrak{X}(T_{H}) := \Hom_{\Z}(T_{H},\, \C^{\times}) (\simeq \Hom_{\Z}(B_{H},\, \C^{\times}))$, defined as
%        \[
%           \{\chi_{f} \in \mathfrak{X}(T_{H}) : b.f = \chi_{f}(b) \cdot f,\, \forall b \in B_{H} \text{ for some } 0\not= f \in \C(Z)^{(B_{H})} \}.
 %       \]
\item The rank $\rank_H(Z)$ of $Z$ is by definition the rank of the weight lattice $\Lambda_{H}(Z)$. % as a free abelian group.
    \end{enumerate}
\end{defn}

We thank Michel Brion for suggesting a proof of Theorem \ref{thm:main} using the following lemma.

\begin{lemma} \label{lemma:loc-str-spherical}
Let $Z$ be a smooth irreducible $H$-variety such that there exists $z \in Z$ with $H. z$ projective and let $L_{H}$ be a Levi subgroup of $\Stab_{H}(z)$.

The variety $Z$ is $H$-spherical if and only if the normal space $N_{z} := T_{z}Z / T_{z} Y$ is $L_{H}$-spherical.
    In this case, $\rank_{H}(Z) = \rank_{L_{H}}(N_{z})$.
\end{lemma}

\begin{proof}
Set $Y = H.z$ and $P_{H} = \text{Stab}_{H}(z)$. Up to conjugation, we may assume that $B_{H}$ fixes $z$. Let $B_{H}^{-}$ and $P_{H}^{-}$ be the parabolic subgroups opposite to $B_H$ and $P_H$ with respect to $T_H$. We may assume that $L_{H} = P^{-}_{H} \cap P_{H}$.
    Denote by $R^{u}(P_{H})$ and $R^{u}(P^{-}_{H})$ the unipotent radicals of $P_{H}$ and $P^{-}_{H}$, respectively.
    By the local structure theorem \cite[Th\'eor\`eme 1.4]{brion.luna.vust:espaces}, there exists an $L_{H}$-stable smooth affine subvariety $W \subset Z$ such that $z \in W$ and the map
    \[
        R^{u}(P^{-}_{H}) \times W \rightarrow P^{-}_{H} . W \subset Z, \quad (p,\,w) \mapsto p . w
    \]
is an open embedding. This implies that $W$ is $L_{H}$-spherical if and only if $Z$ is $H$-spherical: $B^{-}_{L_H} := B^{-}_{H} \cap L_{H}$ has an open orbit in $W$ if and only if $B_H^- = B_{L_H}^-R^u(P_H^-)$ has a dense orbit in $R^u(P_H^-).W$ and thus in $Z$. Moreover, by Luna's \'etale slice Theorem \cite[p. 97]{luna:slices}, there is an $L_{H}$-equivariant \'etale morphism $S \rightarrow T_{z}W \simeq N_{z}$ from an $L_{H}$-stable affine open neighborhood $S$ of $z \in W$ onto an $L_{H}$-stable open subset of $N_{z}$.
    Therefore $N_{z}$ is $L_{H}$-spherical if and only if so is $W$, or equivalently $Z$ is $H$-spherical (recall that \'etale morphisms are open).

Assume that $Z$ is $H$-spherical, 
%   By \cite[Lemma 7.5]{knop:luna}, $P^{-}_{H}$ is the stabiliser of the union of prime divisors on $Z$ which are not $H$-stable but $B_{H}^{-}$-stable, and not containing $Y$. Therefore by 
and then the local structure theorem for spherical varieties \cite[Theorem 3.2.2]{perrin:sanya} implies that $W$ above satisfies $\rank_{H}(Z) = \rank_{L_{H}}(W) - \rank_H(Y) = \rank_{L_H}(W)$.
% there is an $L_{H}$-stable smooth affine subvariety $W' \subset Z$ such that $\rank_{H}(Z) = \rank_{L_{H}}(N_{z})$,
%    \[
%        R^{u}(P^{-}_{H}) \times W' \rightarrow P^{-}_{H} . W' \subset Z, \quad (p,\,w) \mapsto p. w
%    \]
%    is an open embedding and $W' \cap Y$ is an $L_{H}$-orbit.
%    In fact, 
Since $Y$ is projective, $W \cap Y$ is affine and projective thus consists of the single $L_{H}$-fixed point $z$.
%    Since $R^{u}(P^{-}_{H}) . z' = P^{-}_{H} . (W' \cap Y) = (P^{-}_{H} . W') \cap Y$ is open in $Y$, $z'=z$.
Furthermore, since $L_{H}$ has an open orbit in $W$, we have a $L_{H}$-isomorphism $W \simeq T_{z}W$ (see \cite[Corollaire III.1.2]{luna:slices}), hence $N_{z} (\simeq T_{z}W)$ is an $L_{H}$-spherical variety of rank $\rank_H(Z)$.
\end{proof}

A classification of spherical representations is well known, see for example \cite{leahy:classification}.
In Table \ref{table-spherical}, following \cite[Table 1]{leahy:classification}, we give a list of some spherical representations which are used in the sequel.
Here, we denote by $V_{H}$ the standard representation of a simple Lie group $H$ of classical type, and by $\C^{1}$ a nontrivial irreducible representation of $\C^{\times}$.
Also, $\Stt_{\text{Spin}_{10}}$ means the spin representation of $\text{Spin}_{10}$.

\begin{table}[ht]
\begin{tabular}{c|c|c}
$H$ & Spherical $H$-representation $V$ & $\rank_{H}(V)$ \\
\hline
$\SL_{n} \times \C^{\times}$ & $V_{\SL_{n}} \otimes \C^{1}$ & 1 \\
$\SL_{m} \times \SL_{n} \times \C^{\times}$ & $V_{\SL_{m}} \otimes V_{\SL_{n}} \otimes \C^{1}$ & $\min(m,\,n)$ \\
$\SL_{n} \times \C^{\times}$ & $(\Sym^{2} V_{\SL_{n}}) \otimes \C^{1}$ & $n$ \\
$\SL_{n} \times \C^{\times}$ & $(\bigwedge^{2}V_{\SL_{n}}) \otimes \C^{1}$ & $\lfloor n/2 \rfloor$ \\
$\SO_{n} \times \C^{\times}$ & $V_{\SO_n} \otimes \C^{1}$ & 2 \\
$\text{Spin}_{10} \times \C^{\times}$ & $\Stt_{\text{Spin}_{10}} \otimes \C^{1}$ & 2 \\
\hline
\end{tabular}

~\\
\centering
~\\
\caption{\label{table-spherical} Some spherical representations and their rank.}
\end{table}

\subsection{\texorpdfstring{$B$}{B}-stable double lines}

We discuss $B$-stable conics on $X$.
Let $V$ be the irreducible representation associated to the fundamental weight $\varpi$.
We denote by $v_{\varpi} \in V$ a nonzero weight vector with weight $\varpi$, and choose the base point $o := [v_{\varpi}]$ so that $X = G . o \subset \bP(V)$.

Let $C$ be a conic on $\bP(V)$, i.e. a closed subscheme of $\bP (V)$ with Hilbert polynomial $2t+1$.
Then there are three possibilities: $C$ is smooth, $C$ is a union of two distinct lines, or $C$ is a non-reduced quadric on a plane in $\bP(V)$.
In the last two cases, we say that $C$ is a \emph{reducible} conic and a \emph{double line}, respectively. 
In any case, $C$ is contained in a unique plane $\Pi \subset \bP(V)$ as a closed subscheme.

\begin{prop} \label{prop:Hilb-irreducible}
    Let $C$ be a $B$-stable conic on $\bP(V)$.
    \begin{enumerate}
        \item $C$ is a double line on a $B$-stable plane $\Pi := \bP (v_{\varpi}, \, v_{\varpi - \delta},\, v_{\varpi - \delta - \alpha})$ supported on a $B$-stable line $\ov{C} := \bP (v_{\varpi}, \, v_{\varpi - \delta})$ for some $\alpha \in \Delta$ satisfying $\langle \alpha^{\vee},\, \delta \rangle < 0$.
        \item $C$ is a closed subscheme of $X$, and $\Hilb_{2}(X)$ is irreducible.
    \end{enumerate}
\end{prop}
\begin{proof}
(1) Let $\Pi \subset \bP(V)$ be the plane containing $C$.
    Since $C$ is $B$-stable, $\Pi$ is also $B$-stable, and so $\Pi = \bP (v_{\varpi}, \, v_{\varpi - \delta},\, v_{\varpi - \delta - \alpha})$ for some $\alpha \in \Delta$ satisfying $\langle \alpha^{\vee},\, \delta \rangle < 0$.
    Observe that a defining equation of $C$ on $\Pi$ is an element of a $B$-representation $H^{0}(\Pi,\, \cO_{\Pi}(2))$, the space of quadratic equations.
    If $[x_{\varpi}:x_{\varpi - \delta}:x_{\varpi - \delta - \alpha}] = [x_{\varpi} v_{\varpi}+x_{\varpi - \delta}v_{\varpi - \delta}+x_{\varpi - \delta - \alpha}v_{\varpi - \delta - \alpha}]$ is a homogeneous coordinate, then $H^{0}(\Pi,\, \cO_{\Pi}(2))$ is generated by weight vectors $x_{\omega}x_{\omega'}$ with weight $-(\omega + \omega')$.
    It is straightforward to show that their weights are pairwise different, and so a defining equation of $C$ is of form $x_{\omega} x_{\omega'}$.
    Since its support $C^{\text{red}}$ is $B$-stable, and since $\ov{C} := \bP (v_{\varpi}, \, v_{\varpi - \delta})$ is a unique $B$-stable line on $X$, $C$ is defined by an equation $x_{\varpi - \delta - \alpha}^{2}$.
    That is, $C$ is a double line lying on $\Pi$ with $C^{\text{red}} = \ov{C}$.

(2) In the normalization of $\Hilb_{2}(X)$, let $\cH$ be the Zariski closure of the locus of smooth conics, endowed with the reduced scheme structure.
    Since $P$ is a maximal parabolic, $\cH$ is irreducible (see \cite{thomsen:irreducibility,kim.pandharipande:connectedness,perrin:courbes}).
    By the Borel fixed point theorem, each component of the normalization of $\Hilb_{2}(X)$ contains a $B$-fixed point, which corresponds to a $B$-stable conic on $X$.
    Therefore it suffices to show that $C$ represents a point in $\cH$.

    Assume that $\delta$ represents an end node in the Dynkin diagram of $\fg$.
    That is, there is unique $\alpha \in \Delta$ satisfying $\langle \alpha^{\vee},\, \delta \rangle < 0$.
    Then $C$ is a unique $B$-stable conic on $\bP(V)$, hence by the Borel fixed point theorem, it represents a point in $\cH$.

    Thus we may assume that $\delta$ is not an end node, in particular $G \not= G_{2}$.
    If $\langle \alpha^{\vee},\, \delta \rangle = -1$, then the plane $\Pi$ is contained in $X$.
    In fact, if $\Pi \not\subset X$, then since the ideal of $X$ in $\bP(V)$ is generated by quadratic equations, the scheme-theoretic intersection $\Pi \cap_{\text{sch}} X$ is a $B$-stable conic $C$, which is a contradiction since $[v_{\varpi - \delta - \alpha}] = [n_{\alpha}. v_{\varpi - \delta}] \in X$.
    Here, $n_{\alpha}$ means an element of $N(T)$ representing $s_{\alpha} \in W \simeq N_{G}(T)/T$, and $s_{\alpha}(\varpi - \delta) = \varpi - \delta - \alpha$ since $\langle \alpha^{\vee},\, \delta \rangle = -1$.
    Therefore every conic on $\Pi$ is a closed subscheme of $X$.
    Since the space of conics on $\Pi$ is irreducible (in fact, $\simeq \bP^{5}$), $C$ represents a point in $\cH$.

    Now assume that $\langle \alpha^{\vee},\, \delta \rangle \not= -1$.
    Since $G \not=G_{2}$, $\langle \alpha^{\vee},\, \delta \rangle = -2$.
    Thus the orbit closure of $[v_{\varpi - \delta}]$ under the action of the root subgroup $\exp(\fg_{-\alpha})$ is a smooth conic contained in a plane $\bP(v_{\varpi - \delta},\, v_{\varpi - \delta - \alpha},\, v_{\varpi - \delta - 2\alpha})$.
    Denote this smooth conic by $C^{\text{sm}} \in \cH$.
    Observe that $C^{\text{sm}}$ is lying on a $B$-stable 3-plane $\bP(v_{\varpi},\, v_{\varpi - \delta},\, v_{\varpi - \delta - \alpha},\, v_{\varpi - \delta - 2\alpha})$.
    Since the only $B$-stable conic on the 3-plane is $C$, it represents a point in the closure of $B.C^{\text{sm}}$ in $\cH$.
\end{proof}

From now on, $C$ denotes a $B$-stable conic on $\bP(V)$ (hence on $X$), and we keep the notation of Proposition \ref{prop:Hilb-irreducible}.

For a scheme $X'$ and its closed subscheme $X''$, we denote by $\cI_{X''/X'}$ the associated ideal sheaf, and by $N^{*}_{X''/X'} := \cI_{X''/X'}|_{X''}$ the conormal sheaf.
In the case where $N^{*}_{X''/X'}$ is locally free, we denote by $N_{X''/X'}$ its dual $\cO_{X''}$-module, and we call it the normal bundle of $X''$ in $X'$.
Since $C$ is a local complete intersection (being a hypersurface in $\Pi$), the conormal sheaves $N^{*}_{C/\Pi}$ and $N^{*}_{C/X}$ are locally free, and hence the normal bundles are well defined.
In fact, since $C$ is a quadric in $\Pi$, $N^{*}_{C/\Pi}|_{\ov{C}} \simeq \cO_{\bP^{1}}(-2)$.
Moreover, since $(\cI_{\ov{C}/C})^{2} = 0$, $\cI_{\ov{C}/C}$ is indeed an $\cO_{\ov{C}}$-module, isomorphic to $\cO_{\bP^{1}}(-1)$, and hence $N_{\ov{C}/C} \simeq \cO_{\bP^{1}}(1)$.

For a subset $\Sigma \subsetneq \Delta$, let $P_{\Sigma}$ be the standard parabolic subgroup containing $B$ and generated by $\Delta \setminus \Sigma$.
In other words, $\Delta_{P_{\Sigma}} = \Delta \setminus \Sigma$.
For example, $P = P_{\delta}$.
For $\beta \in \Delta$, put $N(\beta) := \{\gamma \in \Delta : \langle \gamma^{\vee},\, \beta \rangle < 0\}$ and
$\ov{Q} := \text{Stab}_{G}(\ov{C}) (= P_{N(\delta)})$.
Define $Q:= \text{Stab}_{G}(C) (= \ov{Q} \cap \text{Stab}_{G}(\Pi))$.
The standard Levi subgroups of $P$, $\ov{Q}$ and $Q$ are denoted by $P^{\text{Levi}}$, $\ov{L}$ and $L$, respectively.
The Lie algebras of $\ov{Q}$, $Q$, $P^{\text{Levi}}$, $\ov{L}$ and $L$ are denoted by $\ov{\fq}$, $\fq$, $\fp^{\text{Levi}}$, $\ov{\fl}$ and $\fl$, respectively.
In this notation, it is easy to show that
\[
    \text{Stab}_{G}(\Pi) = \left\{
    \begin{array}{cc}
        P_{(N(\delta) \setminus \{\alpha\}) \cup (N(\alpha) \setminus \{\delta\})} & \text{if }|\delta| = |\alpha|, \\
         P_{N(\delta) \cup (N(\alpha) \setminus \{\delta\})} & \text{if }|\delta| > |\alpha|, \\
         P_{(N(\delta) \setminus \{\alpha\}) \cup N(\alpha)} & \text{if }|\delta| < |\alpha|,
    \end{array}
    \right.
\]
and so
\[
    Q = \left\{
    \begin{array}{cc}
        P_{N(\delta) \cup (N(\alpha) \setminus \{\delta\})} & \text{if }|\delta| \ge |\alpha|, \\
         P_{N(\delta) \cup N(\alpha)} & \text{if }|\delta| < |\alpha|,
    \end{array}
    \right.
\]
and
\[
    P \cap Q = P_{N(\delta) \cup N(\alpha)}.
\]
Here, $|\cdot|$ means the length of a given root with respect to the Killing form $(\,,\,)$.
Let $\fr$ be the Lie algebra of the standard Levi subgroup of $P\cap Q$.
Then we have $\fr = \fp^{\text{Levi}} \cap \fl$ and
\[
    [\fl,\, \fl] = \left\{ \begin{array}{cc}
        {[\fr,\, \fr]} \oplus \mathfrak{sl}_{2}(\delta) & \text{if }|\delta| \ge |\alpha|, \\
        {[\fr,\, \fr]} & \text{if }|\delta| < |\alpha|,
    \end{array} \right.
\]
where $\mathfrak{sl}_{2}(\delta)$ is the 3-dimensional subalgebra generated by $\pm \delta$.
Finally, for a standard Levi subgroup $L'$, a 1-dimensional $L'$-representation where its center acts via the character $\chi$ is denoted by $\C_{\chi}$.
By a slight abuse of notation, if $\chi$ is the restriction of a character $\tilde{\chi}$ of $T$, we write the $L'$-representation $\C_{\chi}$ as $\C_{\tilde{\chi}}$.

\subsection{Infinitesimal structure of Hilbert schemes of conics}

Next, we introduce natural exact sequences involving the normal bundle $N_{C/X}$, and then show the smoothness of $\Hilb_{2}(X)$.

The following sequences can be also found in \cite[Lemma~A.2.4]{kuznetsov.prokhorov.shramov:hilbert} where a more general setup is discussed. We thank Kiryong Chung for pointing it out.

\begin{prop}\label{prop:basic-seq-for-double-line}
    There are natural exact sequences
    \begin{equation} \label{eqn:normal-bundle-double-in-X}
    0 \rightarrow  N_{C/X} |_{\ov{C}} \otimes_{\ov{C}} \cI_{\ov{C}/C} \rightarrow N_{C/X} \rightarrow N_{C/X}|_{\ov{C}} \rightarrow 0
    \end{equation}
    of $\cO_{C}$-modules, induced by the defining sequence for $\ov{C} \subset C$, and
    \begin{equation} \label{eqn:normal-seq-double}
    0 \rightarrow N_{\ov{C}/C} \rightarrow N_{\ov{C}/X} \rightarrow N_{C/X}|_{\ov{C}} \rightarrow N_{C/\Pi}|_{\ov{C}} \rightarrow 0
    \end{equation}
    of $\cO_{\ov{C}}$-modules, induced by the conormal sequence for $\ov{C} \subset C \subset X$.
\end{prop}

\begin{proof}
Recall the exact sequence of $\cO_{C}$-modules
\[
    0 \rightarrow \cI_{\ov{C}/C} \rightarrow \cO_{C} \rightarrow \cO_{\ov{C}} \rightarrow 0.
\]
Since $\cI_{\ov{C}/C}$ is an $\cO_{\ov{C}}$-module, by taking the tensor product with a locally free $\cO_{C}$-module $N_{C/X}$, we obtain the sequence (\ref{eqn:normal-bundle-double-in-X}).

On the other hand, we have the conormal exact sequence over $\ov{C}$
\[
    N^{*}_{C/X}|_{\ov{C}} \rightarrow N^{*}_{\ov{C}/X} \rightarrow N^{*}_{\ov{C}/C} (\simeq \cI_{\ov{C}/C}) \rightarrow 0.
\]
Let $\cK$ be the kernel of the morphism $N^{*}_{C/X}|_{\ov{C}} \rightarrow N^{*}_{\ov{C}/X}$.
Then $\cK$ is an invertible $\cO_{\ov{C}}$-module.
Thus by comparing the determinant bundles, we see that there is a natural isomorphism
\[
    \cK \simeq \det \left( N^{*}_{C/X}|_{\ov{C}} \right) \otimes \det \left( N_{\ov{C}/X} \right) \otimes N^{*}_{\ov{C}/C}.
\]
Furthermore, the adjunction formula implies that $\omega_{\Pi} \otimes N_{C/\Pi}$ and $\omega_{X} \otimes \det N_{C/X}$ are naturally isomorphic to the dualizing sheaf of $C$ as $\cO_{C}$-modules.
It implies that
\begin{align*}
    &\det \left( N^{*}_{C/X}|_{\ov{C}} \right) \otimes \det \left( N_{\ov{C}/X} \right) \\
    &\simeq \omega_{X}|_{\ov{C}} \otimes (\omega_{\Pi}  \otimes N_{C/\Pi})^{*}|_{\ov{C}} \otimes \det \left( N_{\ov{C}/X} \right)\\
    &\simeq (T^{*}\ov{C}) \otimes \det \left( N^{*}_{\ov{C}/X} \right) \otimes (\det T\Pi)|_{\ov{C}}  \otimes N^{*}_{C/\Pi}|_{\ov{C}} \otimes \det \left( N_{\ov{C}/X} \right)\\
    &\simeq N_{\ov{C}/\Pi} \otimes N^{*}_{C/\Pi}|_{\ov{C}}.
\end{align*}
Note that the quotient map $\cO_{\Pi} \rightarrow \cO_{C}$ induces a surjective morphism $\cI_{\ov{C}/\Pi} \rightarrow \cI_{\ov{C}/C}$, which again induces an isomorphism $N^{*}_{\ov{C}/\Pi} \rightarrow N^{*}_{\ov{C}/C}$.
Thus we have a natural identification $\cK \simeq N^{*}_{C/\Pi}|_{\ov{C}} (\simeq \cO_{\ov{C}}(-2))$, and hence the sequence (\ref{eqn:normal-seq-double}).
\end{proof}

\begin{cor} \label{cor:hilb-is-smooth}
    $H^{1}(C,\, N_{C/X}) = 0$.
    In particular, $\Hilb_{2}(X)$ is smooth and its tangent space at $C$ is isomorphic to $H^{0}(C,\, N_{C/X})$.
\end{cor}
\begin{proof}
    Since $X$ is homogeneous, the normal bundle $N_{\ov{C}/X}$ is globally generated, and hence $H^{1}(\ov{C},\, N_{\ov{C}/X}) = H^{1}(\ov{C},\,N_{\ov{C}/X} \otimes \cI_{\ov{C}/C}) = 0$.
    Thus by the sequence (\ref{eqn:normal-seq-double}),
    \[
        H^{1}(\ov{C},\,N_{C/X}|_{\ov{C}}) = H^{1}(\ov{C},\,N_{C/X}|_{\ov{C}} \otimes \cI_{\ov{C}/C}) = 0,
    \]
    and hence the sequence (\ref{eqn:normal-bundle-double-in-X}) implies that $H^{1}(C,\,N_{C/X}) = 0$.
    Since $C \subset X$ is a local complete intersection, by \cite[Theorem I.2.8, Lemma I.2.12.1 and Proposition I.2.14.2]{kollar:rational}, $\Hilb_{2}(X)$ is smooth at the point $C$ and its tangent space is $H^{0}(C,\, N_{C/X})$.
    Finally, since $C$ is an arbitrary $B$-stable conic, the Borel fixed point theorem implies that $\Hilb_{2}(X)$ is smooth.
\end{proof}

    Consider the normal spaces at $C$ and $\ov{C}$ of the $G$-orbits $G.C$ and $G.\ov{C}$ in their respective Hilbert schemes:
    \[
        N_{C} := N_{G . C / \Hilb_{2}(X),\, C}, \quad N_{\ov{C}} := N_{G . \ov{C} / \Hilb_{1}(X),\, \ov{C}}.
    \]
\begin{lemma} \label{lemma:computing-normal-sp}
We have an $L$-equivariant isomorphism
        \[
            N_{C} \simeq \frac{\Sym^{2}H^{0}(\ov{C},\,N_{\ov{C}/\Pi}) \oplus 
            N_{\ov{C}} \oplus H^{0}(\ov{C},\, N_{\ov{C}/X} \otimes N^{*}_{\ov{C}/\Pi}) }{\C_{0} \oplus (\ov{\fq} / \fq)}.
        \]
\end{lemma}

\begin{remark}
If $\delta$ is long, then $\Hilb_{1}(X)$ is homogeneous by \cite{landsberg.manivel:on} (or by Theorem \ref{thm:small-d=1}), and hence $N_{\ov{C}} = 0$.
\end{remark}

\begin{proof}%[Proof of Lemma \ref{lemma:computing-normal-sp}]
    Since $X$ is homogeneous, $N_{\ov{C}/X}$ is globally generated, and hence
    \[
        H^{1}(\ov{C},\, N_{\ov{C}/X}) = H^{1}(\ov{C},\, N_{\ov{C}/X} \otimes \cI_{\ov{C}/C}) = 0.
    \]
    In particular, $\Hilb_{1}(X)$ is smooth, $T_{\ov{C}} \Hilb_{1}(X) = H^{0}(\ov{C},\, N_{\ov{C}/X})$ and $N_{\ov{C}}$ is well defined.
    By the sequence (\ref{eqn:normal-seq-double}) in Proposition \ref{prop:basic-seq-for-double-line},
    \[
        H^{1}(\ov{C},\, N_{C/X}|_{\ov{C}}) = H^{1}(\ov{C},\, N_{C/X}|_{\ov{C}} \otimes \cI_{\ov{C}/C}) = 0,
    \]
    and using both sequences of Proposition \ref{prop:basic-seq-for-double-line}, we have a diagram
    \begin{center}
        \begin{tikzcd}[column sep=tiny, row sep=small]
&0 \arrow[d]&&0 \arrow[d]&\\
&H^{0}(\ov{C},\, N_{\ov{C}/C} \otimes \cI_{\ov{C}/C}) \arrow[d] &&H^{0}(\ov{C},\, N_{\ov{C}/C}) \arrow[d]&\\
&H^{0}(\ov{C},\, N_{\ov{C}/X} \otimes \cI_{\ov{C}/C}) \arrow[d] && H^{0}(\ov{C},\, N_{\ov{C}/X}) \arrow[d] &\\
0 \arrow[r] &H^{0}(\ov{C},\, N_{C/X} \otimes \cI_{\ov{C}/C}) \arrow[r] \arrow[d] &H^{0}(C,\, N_{C/X}) \arrow[r] &H^{0}(\ov{C},\, N_{C/X}|_{\ov{C}})  \arrow[r]\arrow[d]&0\\
&H^{0}(C,\, N_{C/\Pi}|_{\ov{C}} \otimes \cI_{\ov{C}/C}) \arrow[d]&&H^{0}(\ov{C},\, N_{C/\Pi}|_{\ov{C}}) \arrow[d]&\\
&0&&0&
        \end{tikzcd}
    \end{center}
of $L$-representations with exact row and columns. As $L$-representations, we get
\begin{align*}
    N_{C}& \simeq \frac{H^{0}(\ov{C},\, N_{C/X}|_{\ov{C}}) \oplus H^{0}(\ov{C},\, N_{C/X} \otimes \cI_{\ov{C}/C})}{T_{C} (G . C)} \\
    &\simeq \frac{\{H^{0}(\ov{C},\, N_{C/\Pi}|_{\ov{C}}) \oplus H^{0}(\ov{C},\, N_{\ov{C}/X})\} \oplus H^{0}(\ov{C},\, N_{C/X} \otimes \cI_{\ov{C}/C})}{H^{0}(\ov{C},\, N_{\ov{C}/C}) \oplus T_{C} (G . C)} \\
    & \simeq \frac{H^{0}(\ov{C},\, N_{C/\Pi}|_{\ov{C}}) \oplus (N_{\ov{C}} \oplus \fg/\ov{\fq}) \oplus H^{0}(\ov{C},\, N_{C/X} \otimes \cI_{\ov{C}/C})}{H^{0}(\ov{C},\, N_{\ov{C}/C}) \oplus (\fg / \fq)} \\
    & \simeq \frac{H^{0}(\ov{C},\, N_{C/\Pi}|_{\ov{C}}) \oplus N_{\ov{C}} \oplus H^{0}(\ov{C},\, N_{C/X} \otimes \cI_{\ov{C}/C})}{H^{0}(\ov{C},\, N_{\ov{C}/C}) \oplus (\ov{\fq} / \fq)} \\
    & \simeq \frac{H^{0}(\ov{C},\, N_{C/\Pi}|_{\ov{C}}) \oplus N_{\ov{C}} \oplus H^{0}(\ov{C},\, N_{\ov{C}/X} \otimes \cI_{\ov{C}/C}) \oplus H^{0}(\ov{C},\, N_{C/\Pi}|_{\ov{C}} \otimes \cI_{\ov{C}/C})}{H^{0}(\ov{C},\, N_{\ov{C}/C}) \oplus (\ov{\fq} / \fq) \oplus H^{0}(\ov{C},\, N_{\ov{C}/C} \otimes \cI_{\ov{C}/C})}.
\end{align*}
Since $N_{\ov{C}/C} (\simeq \cO_{\ov{C}}(1))$ and $\cI_{\ov{C}/C}(\simeq \cO_{\ov{C}}(-1))$ are dual to each other as $\cO_{\ov{C}}$-modules, $H^{0}(\ov{C},\, N_{\ov{C}/C} \otimes \cI_{\ov{C}/C}) (\simeq \C)$ is the trivial $L$-representation $\C_{0}$.
    Moreover, the natural morphisms of $\cO_{\Pi}$-modules
    \[
        \cI_{\ov{C}/C} \otimes \cI_{\ov{C}/C} \twoheadleftarrow \cI_{\ov{C}/\Pi} \otimes \cI_{\ov{C}/\Pi} \xrightarrow{\simeq} \cI_{C/\Pi}
    \]
    induce an isomorphism $\cI_{\ov{C}/C} \otimes \cI_{\ov{C}/C} \xrightarrow{\simeq} N^{*}_{C/\Pi}|_{\ov{C}}$ of $\cO_{\ov{C}}$-modules.
    Thus we have natural identifications $N_{\ov{C}/C} \otimes N_{\ov{C}/C} \simeq N_{C/\Pi}|_{\ov{C}}$ and $N_{C/\Pi}|_{\ov{C}} \otimes \cI_{\ov{C}/C} \simeq N_{\ov{C}/C}$, which imply
    \[
        \text{Sym}^{2}H^{0}(\ov{C},\, N_{\ov{C}/C}) \simeq H^{0}(\ov{C},\, N_{C/\Pi}|_{\ov{C}})
    \]
    and
    \[
        H^{0}(\ov{C},\, N_{C/\Pi}|_{\ov{C}} \otimes \cI_{\ov{C}/C}) \simeq H^{0}(\ov{C},\, N_{\ov{C}/C})
    \]
    as $L$-representations, respectively.
    Finally, since $N^{*}_{\ov{C}/\Pi} \simeq \cI_{\ov{C}/C}$ as $\cO_{\ov{C}}$-modules, the formula follows.
\end{proof}

%%%%%%%%%%%%%%%%%%%%%%%%%%%%%%%%%%%%%%%%%%%%%%%%%%%%%%
%\input{representation}

\section{Representation theoretic method}
\label{section:representation}

Now we prove Theorem \ref{thm:main} via a representation theoretic method, based on Lemma \ref{lemma:loc-str-spherical} and Lemma \ref{lemma:computing-normal-sp}.
Namely, from Section \ref{subsection:The case of end nodes} to Section \ref{subsection:remaining-cases}, we consider the case where $\delta$ is long, compute $N_{C}$ using Lemma \ref{lemma:computing-normal-sp}, and show that $N_{C}$ is $L$-spherical by Table \ref{table-spherical}.
Then it follows that $\Hilb_{2}(X)$ is $G$-spherical by Lemma \ref{lemma:loc-str-spherical} and Corollary \ref{cor:hilb-is-smooth}.
The case where $\delta$ is short is considered in Section \ref{subsection:short-roots}.

Before going into more detail, we explain how to compute the representations given in Lemma \ref{lemma:computing-normal-sp}.
We start with $\text{Sym}^{2} H^{0}(\ov{C},\, N_{\ov{C}/\Pi})$ which is easier to compute.

\begin{prop} \label{prop:normal-of-line-in-plane}
    Let $[x:y:z]$ be the homogeneous coordinate of $\Pi$, given by $[x:y:z] = [x v_{\varpi} + y v_{\varpi - \delta} + z v_{\varpi-\delta - \alpha}]$.
    As $L$-representations,
    \[
        H^{0}(\Pi,\, \cO_{\Pi}(\ov{C})) = \C \langle x/z,\, y/z,\, 1 \rangle
    \]
    where $x/z$ and $y/z$ are weight vectors of weights $-(\delta + \alpha)$ and $-\alpha$, respectively, and
    \[
        H^{0}(\ov{C},\, N_{\ov{C}/\Pi}) \simeq \C\langle x/z,\, y/z \rangle \subset H^{0}(\Pi,\, \cO_{\Pi}(\ov{C})).
    \]
\end{prop}
\begin{proof}
    The expression of $H^{0}(\Pi,\, \cO_{\Pi}(\ov{C}))$ follows from the definition of the line bundle $\cO_{\Pi}(\ov{C})$.
    For the expression of $H^{0}(\ov{C},\, N_{\ov{C}/\Pi})$, observe that $N_{\ov{C}/\Pi} \simeq \cO_{\Pi}(\ov{C})|_{\ov{C}}$.
    Then by applying the tensor product with $\cO_{\Pi}(\ov{C})$ to the short exact sequence
    \[
        0 \rightarrow \cI_{\ov{C}/\Pi} \rightarrow \cO_{\Pi} \rightarrow \cO_{\ov{C}} \rightarrow 0,
    \]
    we see that $H^{0}(\ov{C},\, N_{\ov{C}/\Pi}) \simeq H^{0}(\Pi,\, \cO_{\Pi}(\ov{C}))/\C_{0}$.
    Finally, note that the $L$-action preserves $\C\langle x,\,y\rangle$ and $\C \cdot z$, since it is the standard Levi subgroup of $\Stab_{G}(\ov{C}) \cap \Stab_{G}(\Pi)$.
\end{proof}

\begin{cor} \label{coro:sym2-normal-of-line-in-plane}
$\Sym^{2} H^{0}(\ov{C},\, N_{\ov{C}/\Pi})$ is isomorphic to
    \[
        \left\{ \begin{array}{cc}
            \C_{-(\delta+2\alpha)} \otimes \mathfrak{sl}_{2}(\delta) & \text{if }|\delta| \ge |\alpha|, \\
             \C_{-2(\delta+\alpha)} \oplus \C_{-(\delta+2\alpha)} \oplus \C_{-2\alpha}& \text{if }|\delta| < |\alpha|
        \end{array} \right.
    \]
    as an $L$-representation.
    Here, $\mathfrak{sl}_{2}(\delta)$ means the adjoint representation of $\mathfrak{sl}_{2}(\delta)$.
\end{cor}

\begin{proof}%[Proof of Corollary \ref{coro:sym2-normal-of-line-in-plane}]
    Recall that
    \[
    Q = \left\{
    \begin{array}{cc}
        P_{N(\delta) \cup (N(\alpha) \setminus \{\delta\})} & \text{if }|\delta| \ge |\alpha|, \\
         P_{N(\delta) \cup N(\alpha)} & \text{if }|\delta| < |\alpha|,
    \end{array}
    \right.
\]
and so $\delta$ is a root of $L$ if and only if $|\delta| \ge |\alpha|$.
If $\delta$ is not a root of $L$, \emph{i.e.} when $|\delta| < |\alpha|$, then the only simple factor of $L$ acting non-trivially on $\Pi$ is its center.
Thus to obtain the decomposition, it suffices to compute the weights, which follows from Proposition \ref{prop:normal-of-line-in-plane}.
If $\delta$ is a root of $L$, \emph{i.e.} when $|\delta| \ge |\alpha|$, then $\mathfrak{sl}_{2}(\delta)$ is a simple factor of $\fl$, and the simple factors of $\fl$ different from the center of $\fl$ and $\mathfrak{sl}_{2}(\delta)$ act trivially on $\Pi$.
Since $\mathfrak{sl}_{2}(\delta)$ acts non-trivially on $H^{0}(\ov{C},\, N_{\ov{C}/\Pi}) \simeq \C \langle x/z,\, y/z\rangle$, we can write
\[
    \Sym^{2} H^{0}(\ov{C},\, N_{\ov{C}/\Pi})\simeq \C_{\chi} \otimes \mathfrak{sl}_{2}(\delta)
\]
where $\chi$ is a character of the center of $L$.
In fact, by Proposition \ref{prop:normal-of-line-in-plane}, the weights of $\Sym^{2} H^{0}(\ov{C},\, N_{\ov{C}/\Pi})$ are $-2(\delta + \alpha)$, $-(\delta+2\alpha)$ and $-2\alpha$, and so $\C_{\chi} = \C_{- (\delta +2\alpha)}$.
\end{proof}

\begin{cor}
\label{coro:sym2-not-normal}
$\Sym^{2} H^{0}(\ov{C},\, N_{\ov{C}/\Pi})$ is $L$-spherical if and only if $|\delta| \ge |\alpha|$ in which case it has rank $2$.
\end{cor}

\begin{proof}
If $|\delta| \ge |\alpha|$ the result follows by Table \ref{table-spherical}. Otherwise, $L$ acts via by characters its center. Since the three characters are linearly dependent, we get that $\Sym^{2} H^{0}(\ov{C},\, N_{\ov{C}/\Pi})$ is not $L$-spherical.
\end{proof}

Computing the term $H^{0}(\ov{C},\, N_{\ov{C}/X} \otimes N^{*}_{\ov{C}/\Pi})$ in Lemma \ref{lemma:computing-normal-sp} is more difficult. For this, consider $\cC_o$ the variety of minimal rational tangents (the so-called VMRT; see \cite{hwang:geometry} for details) in $X$ defined by
\[
    \cC_{o} := \{[T_{o}l] \in \bP(T_{o}X) : l \text{ is a line on $X$ passing through $o$}\},
\]
which can be considered as the space of lines on $X$ through $o$.
Since $X$ is homogeneous, $\cC_{o}$ is smooth, and there is a natural $P \cap \overline{Q}$-equivariant isomorphism
\[
    T_{\ov{C}} \cC_{o} \simeq H^{0}(\ov{C},\, N_{\ov{C}/X} \otimes \cI_{o/ \ov{C}}).
\]
Note that $\cC_{o}$ is $P$-homogeneous if and only if $\delta$ is long, see \cite{landsberg.manivel:on} (or Theorem~\ref{thm:small-d=1}).

\begin{prop} \label{prop:TC-tensor-charcter}
    As an $\fr$-representation,
        \[
            H^{0}(\ov{C},\, N_{\ov{C}/X} \otimes N^{*}_{\ov{C}/\Pi}) \simeq (T_{\ov{C}} \cC_{o}) \otimes \C_{\alpha}.
        \]
\end{prop}
\begin{proof}
    Since $N^{*}_{\ov{C}/\Pi}$ and $\cI_{o/\ov{C}}$ are isomorphic to $\cO_{\bP^1}(-1)$ as $\cO_{\ov{C}}$-modules, we have an isomorphism
    \[
        H^{0}(\ov{C},\, N_{\ov{C}/X} \otimes N^{*}_{\ov{C}/\Pi}) \otimes H^{0}(\ov{C},\, \cI_{o/\ov{C}} \otimes N_{\ov{C}/\Pi}) \rightarrow T_{\ov{C}} \cC_{o}
    \]
    which is $P \cap Q$-equivariant.
    Since $H^{0}(\ov{C},\, \cI_{o/\ov{C}} \otimes N_{\ov{C}/\Pi})$ is of dimension 1, it is enough compute its weight as a $T$-representation.
    Note that
    \[
        H^{0}(\ov{C},\, \cI_{o/\ov{C}} \otimes N_{\ov{C}/\Pi}) \simeq H^{0}(\ov{C},\, N_{\ov{C}/\Pi}) / H^{0}(o,\, N_{\ov{C}/\Pi}|_{o})
    \]
    as $\fr$-representations.
    Since $o = [v_{\varpi}]$, we have
    \[
        T\ov{C}|_{o} \simeq \Hom(\C \cdot v_{\varpi},\, \C\cdot v_{\varpi - \delta}), \quad T\Pi|_{o} \simeq \Hom(\C\cdot v_{\varpi},\, \C\langle v_{\varpi - \delta},\, v_{\varpi - \delta - \alpha}\rangle),
    \]
    hence
    \[
        H^{0}(o,\, N_{\ov{C}/\Pi}|_{o}) \simeq \Hom(\C\cdot v_{\varpi},\, \C\cdot v_{\varpi-\delta - \alpha}),
    \]
    which is of weight $-\delta - \alpha$.
    By Proposition \ref{prop:normal-of-line-in-plane}, the weights of $H^{0}(\ov{C},\, N_{\ov{C}/\Pi})$ are $-\delta-\alpha$ and $-\alpha$, and so the proof follows.
\end{proof}

We recall another useful lemma, which is a special case of Kostant's lemma \cite[Theorem 8.13.3]{wolf:spaces}. For a subset $\Sigma \subset \Delta$, recall the definition of the parabolic subgroup $P_{\Sigma}$ and let $L_\Sigma$ be its standard Levi subgroup and $\fl_{\Sigma}$ be the Lie algebra of $L_\Sigma$. For $\beta \in R$, define $c_\beta : \Delta \to \Z$ by the formula $\beta = \sum_{\alpha \in \Delta}c_\beta(\alpha)\alpha$. For a function $f: \Sigma \rightarrow \Z$, define 
$$\fg(f) = \sum_{\beta \in R, \ c_\beta\vert_\Sigma = f} \fg_{\beta}.$$

\begin{lemma}[{\cite[\S2.3]{landsberg.manivel:on}}] \label{lemma:g(f)-is-irreducible}
Assume that $f : \Sigma \rightarrow \Z_{\le 0}$ is such that $f \not\equiv 0$ and $\fg(f)\not=0$. Then $\fg(f)$ is an irreducible $\fl_{\Sigma}$-representation. 
\end{lemma}

\begin{prop} \label{prop:coro-Kostant}
Assume that $\delta$ is long.
    \begin{enumerate}
        \item
        As $\fl$-representations, we have the decomposition
        \[
            \ov{\fq} / \fq \simeq \bigoplus_{\substack{f: N(\delta) \cup (N(\alpha) \setminus \{\delta\}) \rightarrow \Z_{\le 0},\\ f\not\equiv 0,\, f|_{N(\delta)} \equiv 0}} \fg(f),
        \]
        and each $\fg(f)$ is irreducible (possibly zero).
        
        \item As $\fp^{\text{Levi}} \cap \ov{\fl}$-representations, we have the decomposition
        \[
            T_{\ov{C}}\cC_{o} \simeq \bigoplus_{\beta \in N(\delta)} \fg(\delta \mapsto 0,\, \beta \mapsto -1),
        \]
        and each $\fg(\delta \mapsto 0,\, \beta \mapsto -1)$ is irreducible with highest weight $-\beta$.

        \item As $\fr$-representations, $\fg(\delta \mapsto 0,\, \beta \mapsto -1)$ is irreducible with highest weight $-\beta$ if $\beta \in N(\delta) \setminus \{\alpha\}$, 
while we have the irreducible decomposition
%while the irreducible decomposition of $\fg(\delta \mapsto 0,\, \alpha \mapsto -1)$ is given by
        \[
\fg(\delta \mapsto 0,\, \alpha \mapsto -1)\simeq \fg_{-\alpha} \oplus \bigoplus_{\substack{f: N(\alpha) \cup \{\alpha\} \rightarrow \Z_{\le 0},\\ f(\delta) = 0,\, f(\alpha) = -1,\, f|_{N(\alpha)} \not\equiv 0}} \fg(f).
        \]
    \end{enumerate}
\end{prop}

\begin{proof}
Recall that since $\delta$ is long, $\fl$, $\fp^{\text{Levi}} \cap \ov{\fl}$ and $\fr$ are generated by $\Delta \setminus (N(\delta) \cup (N(\alpha) \setminus \{\delta\}))$, $\Delta \setminus (N(\delta) \cup \{\delta\})$ and $\Delta \setminus (N(\delta) \cup N(\alpha))$, respectively.

(1) Observe that as an $\fl$-representation, $\fg / \fq \simeq (\fg / \ov{\fq}) \oplus (\ov{\fq} / \fq)$.
    By Lemma \ref{lemma:g(f)-is-irreducible}, the decomposition of $\fg/\fq$ in irreducible $\fl$-representations is given by
    \[
        \fg/\fq \simeq \bigoplus_{f : N(\delta) \cup (N(\alpha) \setminus \{\delta\}) \rightarrow \Z_{\le0},\, f \not\equiv 0} \fg(f).
    \]
    A nonzero $\fg(f)$ is a subrepresentation of $\fg/\ov{\fq}$ if and only if $f|_{N(\delta)}\not\equiv0$, proving (1).
    
(2) Since $\delta$ is long, $\cC_{o} \simeq \prod_{\beta \in N(\delta)} P/P_{\delta,\, \beta}$ and each $P/P_{\delta,\, \beta}$ is an irreducible hermitian symmetric space with respect to the $P^{\text{Levi}}$-action, see \cite{landsberg.manivel:on}.
Therefore the highest weight spaces of $T_{\ov{C}}\cC_{o}$ are the root spaces associated to $-\beta$ for $\beta \in N(\delta)$, and the result follows from this and Lemma \ref{lemma:g(f)-is-irreducible}.

(3) By Lemma \ref{lemma:g(f)-is-irreducible}, each $\fg(f)$ for $f: N(\delta) \cup N(\alpha) \rightarrow \Z_{\le 0}$ with $f\not\equiv 0$ is $\fr$-irreducible or trivial.
    If $\beta \in N(\delta) \setminus \{\alpha\}$, then $\fg(\delta\mapsto 0,\, \beta \mapsto -1) = \fg(f)$ for $f : N(\delta) \cup N(\alpha) \rightarrow \Z_{\le 0}$ defined by $f(\beta) = -1,\quad f(\gamma) = 0, \quad \forall \gamma\not=\beta$ .
%   \[
%       \fg(\delta\mapsto 0,\, \beta \mapsto -1) = \fg(f)
%    \]
%    for $f : N(\delta) \cup N(\alpha) \rightarrow \Z_{\le 0}$ defined by
%    \[
%        f(\beta) = -1,\quad f(\gamma) = 0, \quad \forall \gamma\not=\beta.
%    \]

%\substack{}
    On the other hand,
    \begin{align*}
        \fg(\delta \mapsto 0, \, \alpha \mapsto -1) &= \bigoplus_{\substack{f: N(\delta) \cup N(\alpha) \rightarrow \Z_{\le 0},\\ f(\delta) = 0,\, f(\alpha) = -1}} \fg(f) \\
        &= \bigoplus_{\substack{f: N(\alpha) \cup \{\alpha\} \rightarrow \Z_{\le 0}, \\ f(\delta) = 0,\, f(\alpha) = -1}}
%        {f: N(\alpha) \cup \{\alpha\} \rightarrow \Z_{\le 0},\, f(\delta) = 0,\, f(\alpha) = -1} 
        \fg(f).
    \end{align*}
This proves the claim since for $f(N(\alpha)) = 0$, we have $\fg(f) = \fg_{-\alpha}$.   
\end{proof}

Consider the decomposition $\fl = \fz(\fl) \oplus \mathfrak{sl}_{2}(\delta) \oplus [\fr,\, \fr]$ where $\fz(\fl)$ is the center of $\fl$.

\begin{cor} \label{coro:trivial-sl2}
    Assume that $\delta$ is long.
    \begin{enumerate}
\item The action of $\mathfrak{sl}_{2}(\delta)$ is trivial on $H^{0}(\ov{C},\, N_{\ov{C}/X} \otimes N^{*}_{\ov{C}/\Pi})$.
\item The $\fl$-representation $H^{0}(\ov{C},\, N_{\ov{C}/X} \otimes N^{*}_{\ov{C}/\Pi})$ contains a trivial subrepresentation $\C_{0}$, and the irreducible decomposition of $H^{0}(\ov{C},\, N_{\ov{C}/X} \otimes N^{*}_{\ov{C}/\Pi})/ \C_{0}$ is given as follows:
\[
    \bigoplus_{\beta \in N(\delta) \setminus \{\alpha\}} (\fg (\delta \mapsto 0,\, \beta \mapsto -1) \otimes \C_{\alpha}) \oplus \bigoplus_{\substack{f: N(\alpha) \cup \{\alpha\} \rightarrow \Z_{\le 0},\\ f(\delta) = 0,\, f(\alpha) = -1,\, f|_{N(\alpha)} \not\equiv 0}} (\fg(f) \otimes \C_{\alpha}),
\]
where, $\fg (\delta \mapsto 0,\, \beta \mapsto -1)$, $\fg(f)$ and $\C_{\alpha}$ are $\fz(\fl) \oplus [\fr,\,\fr]$-representations.
%Here, $\fg (\delta \mapsto 0,\, \beta \mapsto -1)$ and $\fg(f)$ are regarded as $\fz(\fl) \oplus [\fr,\,\fr]$-representations.
\end{enumerate}
\end{cor}
\begin{proof}
%    Put $H^{0} := H^{0}(\ov{C},\, N_{\ov{C}/X} \otimes N^{*}_{\ov{C}/\Pi})$.
(1) It is enough to prove that the maximal toral subalgebra $[\fg_{\delta},\, \fg_{-\delta}]$ acts trivially on $H^{0}(\ov{C},\, N_{\ov{C}/X} \otimes N^{*}_{\ov{C}/\Pi})$. By Proposition \ref{prop:coro-Kostant}, if $\omega$ is a weight of the $T$-representation $T_{\ov{C}} \cC_{o}$, then $\langle \delta^{\vee},\, \omega \rangle = 1$.
    Since $\langle \delta^{\vee},\,\alpha \rangle = -1$, Proposition \ref{prop:TC-tensor-charcter} implies that $[\fg_{\delta},\, \fg_{-\delta}]$ acts trivially on $H^{0}(\ov{C},\, N_{\ov{C}/X} \otimes N^{*}_{\ov{C}/\Pi})$. 

(2) Follows from Proposition \ref{prop:TC-tensor-charcter}, Proposition \ref{prop:coro-Kostant}, and (1).
\end{proof}

\subsection{The case of end nodes} \label{subsection:The case of end nodes}

In this section, assume that $\delta$ is long, and that it represents an end node in the Dynkin diagram, that is, $N(\delta)$ consists of a single element.
Then there is only one possible choice of $\alpha$, \emph{i.e.} $N(\delta) = \{\alpha\}$.

Under this assumption, we claim that $N_{C}$ is $L$-spherical.
First, suppose that $N(\alpha) = \{\delta\}$, then, in the notation of \cite{bourbaki:elements*78}, $X$ is one of the following:
\begin{itemize}
    \item $X=A_{2}/P_{1}$ with $\alpha= \alpha_{2}$.
    \item $X= A_{2}/P_{2}$ with $\alpha = \alpha_{1}$.
    \item $X=B_{2}/P_{1}$ with $\alpha = \alpha_{2}$ (or equivalently $X = C_{2}/P_{2}$ with $\alpha = \alpha_{1}$).
    \item $X=G_{2}/P_{2}$ with $\alpha = \alpha_{1}$.
\end{itemize}
In this case, $\ov{Q} = Q = P_{\alpha}$, and by Corollary \ref{coro:trivial-sl2},
\[
    H^{0}(\ov{C},\, N_{\ov{C}/X} \otimes N^{*}_{\ov{C}/\Pi}) / \C_{0} = 0.
\]
By Corollary \ref{coro:sym2-normal-of-line-in-plane} and Lemma \ref{lemma:computing-normal-sp}, we see that
\[
    N_{C} \simeq \C_{-(\delta + 2 \alpha)} \otimes \mathfrak{sl}_{2}(\delta),
\]
which is an $L$-spherical variety of rank 2 (by Table \ref{table-spherical}).

Next, assume that $N(\alpha) = \{\delta,\, \gamma\}$ for $\gamma \in \Delta \setminus \{\delta\}$.
Then $X$ is one of the following:
\begin{itemize}
    \item $X= A_{n}/P_{1}$ with $\alpha = \alpha_{2}$ and $\gamma = \alpha_{3}$ ($n \ge 3$).
    \item $X= A_{n}/P_{n}$ with $\alpha = \alpha_{n-1}$ and $\gamma = \alpha_{n-2}$ ($n \ge 3$).
    \item $X= B_{n}/P_{1}$ with $\alpha = \alpha_{2}$ and $\gamma = \alpha_{3}$ ($n \ge 3$).
    \item $X= C_{n}/P_{n}$ with $\alpha = \alpha_{n-1}$ and $\gamma = \alpha_{n-2}$ ($n \ge 3$).
    \item $X= D_{n}/P_{1}$ with $\alpha = \alpha_{2}$ and $\gamma = \alpha_{3}$ ($n \ge 5$).
    \item $X= E_{r}/P_{1}$ with $\alpha = \alpha_{3}$ and $\gamma = \alpha_{4}$ ($6 \le r \le 8$).
    \item $X= E_{r}/P_{r}$ with $\alpha = \alpha_{r-1}$ and $\gamma = \alpha_{r-2}$ ($6 \le r \le 8$).
    \item $X= F_{4}/P_{1}$ with $\alpha = \alpha_{2}$ and $\gamma = \alpha_{3}$.
\end{itemize}
In this case, $\ov{Q} = P_{\alpha}$ and $Q = P_{\alpha,\, \gamma}$, and hence by Proposition \ref{prop:coro-Kostant},
\[
    \ov{\fq}/\fq \simeq \bigoplus_{k < 0} \fg(\alpha \mapsto 0,\, \gamma \mapsto k)
\]
as $L$-representations, and by Corollary \ref{coro:trivial-sl2},
\[
    H^{0}(\ov{C},\, N_{\ov{C}/X} \otimes N^{*}_{\ov{C}/\Pi}) / \C_{0} \simeq \bigoplus_{k < 0} \fg(\delta \mapsto 0,\, \alpha \mapsto -1,\, \gamma \mapsto k) \otimes \C_{\alpha}
\]
as $\fl$-representations.
In the above list of $(X,\, \alpha,\, \gamma)$, one can easily check that $\Delta \setminus \{\delta,\,\alpha\}$ generates a simple subalgebra, and its highest root has coefficient 1 in $\gamma$.
Thus the previous isomorphism becomes
\[
    \ov{\fq}/\fq \simeq \fg(\alpha \mapsto 0,\, \gamma \mapsto -1),
\]
and its highest weight as an $L$-representation is $-\gamma$.
On the other hand, in the list of $X$, the coefficient of $\gamma$ in the highest root of the simple algebra generated by $\Delta \setminus \{\delta\}$ is 1 if $X=A_{n}/P_{1}$, $A_{n}/P_{n}$ or $C_{n}/P_{n}$, and 2 otherwise.
Therefore
\[
    H^{0}(\ov{C},\, N_{\ov{C}/X} \otimes N^{*}_{\ov{C}/\Pi}) / \C_{0} \simeq \fg(\delta \mapsto 0,\, \alpha \mapsto 0,\, \gamma \mapsto -1)
\]
if $X= A_{n}/P_{1}$, $A_{n}/P_{n}$ or $C_{n}/P_{n}$, and
\[\begin{array}{ll}
    H^{0}(\ov{C},\, N_{\ov{C}/X} \otimes N^{*}_{\ov{C}/\Pi}) / \C_{0} \simeq & \fg(\delta \mapsto 0,\, \alpha \mapsto 0,\, \gamma \mapsto -1) \\
    & \oplus (\fg(\delta \mapsto 0,\, \alpha \mapsto -1,\, \gamma \mapsto -2) \otimes \C_{\alpha}) \\
    \end{array}
%    H^{0}(\ov{C},\, N_{\ov{C}/X} \otimes N^{*}_{\ov{C}/\Pi}) / \C_{0} \simeq \fg(\delta \mapsto 0,\, \alpha \mapsto 0,\, \gamma \mapsto -1) \oplus (\fg(\delta \mapsto 0,\, \alpha \mapsto -1,\, \gamma \mapsto -2) \otimes \C_{\alpha})
\]
otherwise.
Then by Corollary \ref{coro:sym2-normal-of-line-in-plane} and Lemma \ref{lemma:computing-normal-sp},
    \[
        N_{C} \simeq \C_{- (\delta + 2 \alpha)} \otimes \mathfrak{sl}_{2}(\delta)
    \]
if $X = A_{n}/P_{1}$, $A_{n}/P_{n}$ or $C_{n}/P_{n}$, and
\[
    N_{C} \simeq (\C_{- (\delta + 2 \alpha)} \otimes \mathfrak{sl}_{2}(\delta)) \oplus (\fg(\delta \mapsto 0,\, \alpha \mapsto -1,\, \gamma \mapsto -2) \otimes \C_{\alpha})
\]
otherwise.
Now we compare $N_{C}$ with Table \ref{table-spherical} case by case.
\begin{enumerate}
    \item If $X = A_{n}/P_{1}$, $A_{n}/P_{n}$ or $C_{n}/P_{n}$, then $N_{C} \simeq \C_{- (\delta + 2 \alpha)} \otimes \mathfrak{sl}_{2}(\delta)$ is an $L$-spherical variety of rank 2.

    \item Assume that $X \not= A_{n}/P_{1}$, $A_{n}/P_{n}$ and $C_{n}/P_{n}$.
    Then the highest weight of the $\fr$-representation $\fg(\delta \mapsto 0,\, \alpha \mapsto -1,\, \gamma \mapsto -2)$ is
    \begin{itemize}
    \item $-\alpha_{2} - 2 \alpha_{3} - \cdots - 2 \alpha_{n}$ if $X = B_{n}/P_{1}$;
    \item $- \alpha_{2} - 2 \alpha_{3} - \cdots - 2 \alpha_{n-2} - \alpha_{n-1} - \alpha_{n}$ if $X = D_{n}/P_{1}$;
    \item $-\alpha_{2}-\alpha_{3} -2 \alpha_{4} - \alpha_{5}$ if $X = E_{r}/P_{1}$;
    \item $-\alpha_{2} - \alpha_{3} - \sum_{i=4}^{r-2} (2 \alpha_{i}) - \alpha_{r-1}$ if $X = E_{r}/P_{r}$; and
    \item $-\alpha_{2} - 2 \alpha_{3}$ if $X = F_{4}/P_{1}$.
    \end{itemize}
    We denote it by $\rho$.
    Observe that the characters of $\fz(\fl)$ given by $-(\delta + 2\alpha)$ and $\rho + \alpha$ are linearly independent.
    \begin{enumerate}
        \item
        If $X = B_{n}/P_{1}$ or $D_{n}/P_{1}$, then
        \[
        \fg(\delta \mapsto 0,\, \alpha \mapsto -1,\, \gamma \mapsto -2) = \fg_{\rho},
        \]
        hence
        \[
            N_{C} \simeq (\C_{- (\delta + 2 \alpha)} \otimes \mathfrak{sl}_{2}(\delta)) \oplus \C_{\rho+\alpha}.
        \]
        Since the characters $- (\delta + 2 \alpha)$ and $\rho + \alpha$ are linearly independent, $N_{C}$ is an $L$-spherical variety of rank 3.
        
        \item 
        If $X = E_{r}/P_{1}$, then the Dynkin diagram of $[\fr,\, \fr]$ is
        \[
            \dynkin[labels*={\alpha_{2}}]A{*} \dynkin[labels*={\alpha_{5},\alpha_{6},\alpha_{r}}]A{**.*}.
        \]
        By computing the Cartan numbers of $\rho$ with the simple roots of $[\fr,\,\fr]$, we see that $\fg(\delta \mapsto 0,\, \alpha \mapsto -1,\, \gamma \mapsto -2)$ as an $[\fr,\,\fr]$-representation is given by
        \[
            \dynkin[labels*={0}]A{*} \otimes \dynkin[labels*={,1,}]A{**.*},
        \]
        \emph{i.e.} the wedge square of the standard representation of the $A_{r-4}$-factor of $[\fr,\,\fr]$.
        Moreover, $\fz(\fl)$ acts on $\fg(\delta \mapsto 0,\, \alpha \mapsto -1,\, \gamma \mapsto -2)$ via the character $\rho$.
        Since $-(\delta + 2 \alpha)$ and $\rho + \alpha$ are linearly independent,
        \[
            N_{C} \simeq (\C_{- (\delta + 2 \alpha)} \otimes \mathfrak{sl}_{2}(\delta)) \oplus (\fg(\delta \mapsto 0,\, \alpha \mapsto -1,\, \gamma \mapsto -2) \otimes \C_{\alpha})
        \]
        is an $L$-spherical variety of rank 3 if $r = 6$ and of rank 4 if $r=7,\,8$.

        \item 
        If $X = E_{r}/P_{r}$, then the Dynkin diagram of $[\fr,\, \fr]$ is
        \begin{align*}
            \dynkin[labels*={\alpha_{1},\alpha_{3}}]A{**} \dynkin[labels*={\alpha_{2}}]A{*} & \text{ if }r=6,\\
            \dynkin[labels*={\alpha_{1},\alpha_{3}, \alpha_{4}, \alpha_{2}}]A{****} & \text{ if }r=7,\\
            \dynkin[edge length=.7cm, labels*={\alpha_{1},\alpha_{3},\alpha_{4},\alpha_{2},\alpha_{5}}]D{*****}& \text{ if }r=8,
        \end{align*}
        and the Cartan numbers of $\rho$ with the simple roots of $[\fr,\,\fr]$ are
        \begin{align*}
            \dynkin[labels*={1,}]A{**} \dynkin A{*} & \text{ if }r=6,\\
            \dynkin[labels*={1,,,}]A{****} & \text{ if }r=7,\\
            \dynkin[edge length=.7cm, labels*={1,,,,}]D{*****}& \text{ if }r=8.
        \end{align*}
        Therefore
        \[
            N_{C} \simeq (\C_{- (\delta + 2 \alpha)} \otimes \mathfrak{sl}_{2}(\delta)) \oplus (\fg(\delta \mapsto 0,\, \alpha \mapsto -1,\, \gamma \mapsto -2) \otimes \C_{\alpha})
        \]
        is an $L$-spherical variety of rank 3 if $r = 6,\,7$ and of rank 4 if $r=8$.

        \item
        If $X=F_{4}/P_{1}$, then $[\fr,\, \fr] = \mathfrak{sl}_{2}(\alpha_{4})$, the 3-dimensional subalgebra generated by $\pm \alpha_{4}$, and it acts on $\fg(\delta \mapsto 0,\, \alpha \mapsto -1,\, \gamma \mapsto -2)$ via the adjoint representation.
        Thus
        \[
            N_{C} \simeq (\C_{- (\delta + 2 \alpha)} \otimes \mathfrak{sl}_{2}(\delta)) \oplus (\mathfrak{sl}_{2}(\alpha_{4}) \otimes \C_{\rho+\alpha})
        \]
        is an $L$-spherical variety of rank 4.
    \end{enumerate}
\end{enumerate}

Now suppose that $|N(\alpha)| = 3$, and write $N(\alpha) = \{\delta,\, \gamma,\,\gamma'\}$.
Then $X$ is one of the following:
\begin{itemize}
    \item $X = D_{4} / P_{1}$ with $\alpha = \alpha_{2}$, $\gamma = \alpha_{3}$ and $\gamma' = \alpha_{4}$.
    \item $X = D_{n} / P_{n-1}$ with $\alpha = \alpha_{n-2}$, $\gamma = \alpha_{n-3}$ and $\gamma' = \alpha_{n}$ ($n \ge 4$).
    \item $X = D_{n} / P_{n}$ with $\alpha = \alpha_{n-2}$, $\gamma = \alpha_{n-3}$ and $\gamma' = \alpha_{n-1}$ ($n \ge 4$).
    \item $X = E_{r} / P_{2}$ with $\alpha = \alpha_{4}$, $\gamma = \alpha_{3}$ and $\gamma' = \alpha_{5}$ ($6 \le r \le 8$).
\end{itemize}
In this case, $\ov{Q} = P_{\alpha}$ and $Q = P_{\alpha,\,\gamma,\,\gamma'}$.
By Proposition \ref{prop:coro-Kostant},
\[
    \ov{\fq}/\fq \simeq \bigoplus_{k<0} \fg(\alpha \mapsto 0,\, \gamma \mapsto k,\, \gamma' \mapsto 0) \oplus \bigoplus_{k'<0} \fg(\alpha \mapsto 0,\, \gamma \mapsto 0,\, \gamma' \mapsto k')
\]
as $L$-representations.
Since the semi-simple subalgebra generated by $\Delta \setminus \{\alpha\}$ is a product of simple factors of type $A$, we see that
\[
    \ov{\fq}/\fq \simeq \fg(\alpha \mapsto 0,\, \gamma \mapsto -1,\, \gamma' \mapsto 0) \oplus \fg(\alpha \mapsto 0,\, \gamma \mapsto 0,\, \gamma' \mapsto -1).
\]
On the other hand, by Corollary \ref{coro:trivial-sl2},
\[
    H^{0}(\ov{C},\, N_{\ov{C}/X} \otimes N^{*}_{\ov{C}/\Pi})/ \C_{0} \simeq \bigoplus_{0\not= (k,\,k') \in \Z_{\le 0}^{\oplus 2}} \fg(\delta \mapsto 0,\, \alpha \mapsto -1,\, \gamma \mapsto k,\, \gamma' \mapsto k') \otimes \C_{\alpha}
\]
as $\fl$-representations.
Since the semi-simple subalgebra generated by $\Delta \setminus \{\delta\}$ is also a product of simple factors of type $A$,
\begin{align*}
    H^{0}(\ov{C},\, N_{\ov{C}/X} \otimes N^{*}_{\ov{C}/\Pi})/ \C_{0} & \simeq \fg(\delta \mapsto 0,\, \alpha \mapsto 0,\, \gamma \mapsto 0,\, \gamma' \mapsto -1) \\
    & \quad \oplus \fg(\delta \mapsto 0,\, \alpha \mapsto 0,\, \gamma \mapsto -1,\, \gamma' \mapsto 0) \\
    &\quad \oplus (\fg(\delta \mapsto 0,\, \alpha \mapsto -1,\, \gamma \mapsto -1,\, \gamma' \mapsto -1) \otimes \C_{\alpha}).
\end{align*}
By Corollary \ref{coro:sym2-normal-of-line-in-plane} and Lemma \ref{lemma:computing-normal-sp},
\[
    N_{C} \simeq (\C_{-(\delta + 2 \alpha)} \otimes \mathfrak{sl}_{2}(\delta)) \oplus (\fg(\delta \mapsto 0,\, \alpha \mapsto -1,\, \gamma \mapsto -1,\, \gamma' \mapsto -1) \otimes \C_{\alpha}).
\]
Observe that the highest weight of the $\fr$-representation $\fg(\delta \mapsto 0,\, \alpha \mapsto -1,\, \gamma \mapsto -1,\, \gamma' \mapsto -1)$ is $- \alpha - \gamma -\gamma'$.
Thus $\fz(\fl)$ acts on it via the character $-\alpha-\gamma-\gamma'$, and the characters of $\fz(\fl)$ given by $-(\delta + 2 \alpha)$ and $(- \alpha - \gamma -\gamma') + \alpha = -\gamma-\gamma'$ are linearly independent.
Furthermore, $\fg(\delta \mapsto 0,\, \alpha \mapsto -1,\, \gamma \mapsto -1,\, \gamma' \mapsto -1)$ as an $[\fr,\,\fr]$-representation in each case is as follows.
\begin{itemize}
    \item
    If $X = D_{4} / P_{1}$, $D_{4}/P_{3}$ or $D_{4}/P_{4}$, then $\fg(\delta \mapsto 0,\, \alpha \mapsto -1,\, \gamma \mapsto -1,\, \gamma' \mapsto -1) = \fg_{-\alpha - \gamma - \gamma'}$.
    Thus $N_{C}$ is an $L$-spherical variety of rank 3.
    
    \item
    If $X = D_{n}/P_{n-1}$ or $D_{n}/P_{n}$ for $n \ge 5$, then the Dynkin diagram of $[\fr,\,\fr]$ is
    \[
        \dynkin[labels*={\alpha_{1},\alpha_{2},\alpha_{n-4}}] A{**.*},
    \]
    and it acts on $\fg(\delta \mapsto 0,\, \alpha \mapsto -1,\, \gamma \mapsto -1,\, \gamma' \mapsto -1)$ via
    \[
        \dynkin[labels*={,,1}] A{**.*}.
    \]
    Thus $N_{C}$ is an $L$-spherical variety of rank 3.

    \item
    If $X = E_{r}/P_{2}$, then the Dynkin diagram of $[\fr,\,\fr]$ is
    \[
        \dynkin[labels*={\alpha_{1}}] A{*} \dynkin[labels*={\alpha_{6},\alpha_{r}}] A{*.*},
    \]
    and it acts on $\fg(\delta \mapsto 0,\, \alpha \mapsto -1,\, \gamma \mapsto -1,\, \gamma' \mapsto -1)$ via
    \[
        \dynkin[labels*={1}] A{*} \otimes \dynkin[labels*={1,}] A{*.*},
    \]
    \emph{i.e.} the tensor product of the standard representations of $\mathfrak{sl}_{2}$ and $\mathfrak{sl}_{r-4}$.
    Thus $N_{C}$ is an $L$-spherical variety of rank 4.
\end{itemize}

\subsection{The case of neighbors of end nodes}

From now on, we consider the case where $\delta$ is long but not an end node, \emph{i.e.} when $|N(\delta)| \ge 2$.
Then there is a long simple root in $N(\delta)$, say $\beta$, and the plane $\bP(v_{\varpi},\, v_{\varpi - \delta},\, v_{\varpi - \delta - \beta})$ is contained in $X$.
Thus we may choose $C$ as the $B$-stable double line on $\bP(v_{\varpi},\, v_{\varpi - \delta},\, v_{\varpi - \delta - \beta})$.
In other words, any long simple root $\beta$ in $N(\delta)$ can be chosen as $\alpha$.

In this section, as before we show that $N_{C}$ is $L$-spherical, under the assumptions that $|N(\delta)| \ge 2$ and that there is a long element $\alpha \in N(\delta)$ which is an end node, \emph{i.e.} $N(\alpha) = \{\delta\}$.
It can be applied to the following cases:
\begin{itemize}
    \item $X = A_{n}/P_{2}$ with $N(\delta) = \{\alpha_{1},\, \alpha_{3}\}$ and $\alpha = \alpha_{1}$ ($n \ge 3$).
    \item $X = A_{n}/P_{n-1}$ with $N(\delta) = \{\alpha_{n-2},\, \alpha_{n}\}$ and $\alpha = \alpha_{n}$ ($n \ge 3$).
    \item $X = B_{n}/P_{2}$ with $N(\delta) = \{\alpha_{1},\, \alpha_{3}\}$ and $\alpha = \alpha_{1}$ ($n \ge 3$).
    \item $X = D_{n}/P_{2}$ with $N(\delta) = \{\alpha_{1},\, \alpha_{3}\}$ and $\alpha = \alpha_{1}$ ($n \ge 5$).
    \item $X = D_{n}/P_{n-2}$ with $N(\delta) = \{\alpha_{n-3},\, \alpha_{n-1},\, \alpha_{n}\}$ and $\alpha = \alpha_{n}$ ($n \ge 4$).
    \item $X = E_{r}/P_{3}$ with $N(\delta) = \{\alpha_{1},\, \alpha_{4}\}$ and $\alpha = \alpha_{1}$ ($6 \le r \le 8$).
    \item $X = E_{r}/P_{4}$ with $N(\delta) = \{\alpha_{2},\, \alpha_{3},\, \alpha_{5}\}$ and $\alpha = \alpha_{2}$ ($6 \le r \le 8$).
    \item $X = E_{r}/P_{r-1}$ with $N(\delta) = \{\alpha_{r-2},\, \alpha_{r}\}$ and $\alpha = \alpha_{r}$ ($6 \le r \le 8$).
    \item $X = F_{4}/P_{2}$ with $N(\delta) = \{\alpha_{1},\, \alpha_{3}\}$ and $\alpha = \alpha_{1}$.
\end{itemize}
In this case, $\ov{Q} = Q = P_{N(\delta)}$.
By Corollary \ref{coro:trivial-sl2},
\[
    H^{0}(\ov{C},\, N_{\ov{C}/X} \otimes N^{*}_{\ov{C}/\Pi})/ \C_{0} \simeq \bigoplus_{\beta \in N(\delta)\setminus \{\alpha\}} \fg(\delta\mapsto 0,\, \beta \mapsto -1) \otimes \C_{\alpha}
\]
as $\fl$-representations, and hence by Corollary \ref{coro:sym2-normal-of-line-in-plane} and Lemma \ref{lemma:computing-normal-sp},
\[
    N_{C} \simeq (\C_{-(\delta + 2 \alpha)} \otimes \mathfrak{sl}_{2}(\delta)) \oplus \bigoplus_{\beta \in N(\delta)\setminus \{\alpha\}} (\fg(\delta\mapsto 0,\, \beta \mapsto -1) \otimes \C_{\alpha}).
\]
Observe that $\fz(\fl)$ acts on the irreducible factors via the characters $-(\delta + 2\alpha)$ and $-\beta + \alpha$, respectively, which are linearly independent.
Now we compare $N_{C}$ with Table \ref{table-spherical} in each case.
\begin{itemize}
    \item If $X = A_{n}/P_{2}$ with $\alpha = \alpha_{1}$, then the Dynkin diagram of $[\fr,\,\fr]$ is
    \[
        \dynkin[labels*={\alpha_{4},\alpha_{n}}] A{*.*},
    \]
    and it acts on $\fg(\delta\mapsto 0,\, \alpha_{3} \mapsto -1)$ via
    \[
        \dynkin[labels*={1,}] A{*.*}.
    \]
    Thus $N_{C}$ is an $L$-spherical variety of rank 3.
    
    \item If $X = A_{n}/P_{n-1}$ with $\alpha = \alpha_{n}$, then the Dynkin diagram of $[\fr,\,\fr]$ is
    \[
        \dynkin[labels*={\alpha_{1},\alpha_{n-3}}] A{*.*},
    \]
    and it acts on $\fg(\delta\mapsto 0,\, \alpha_{n-2} \mapsto -1)$ via
    \[
        \dynkin[labels*={,1}] A{*.*}.
    \]
    Thus $N_{C}$ is an $L$-spherical variety of rank 3.
    
    \item Suppose that $X = B_{n}/P_{2}$ with $\alpha = \alpha_{1}$.
    If $n=3$, then $\fg(\delta\mapsto 0,\, \alpha_{3} \mapsto -1)=\fg_{-\alpha_{3}}$, hence $N_{C}$ is an $L$-spherical variety of rank 3.
    If $n=4$, then $[\fr,\,\fr] = \mathfrak{sl}_{2}(\alpha_{4})$, which acts on $\fg(\delta\mapsto 0,\, \alpha_{3} \mapsto -1)$ via the adjoint representation, hence $N_{C}$ is an $L$-spherical variety of rank 4.
    If $n \ge 5$, then the Dynkin diagram of $[\fr,\,\fr]$ is
    \[
        \dynkin[labels*={\alpha_{4},,\alpha_{n}}] B{*.**},
    \]
    and it acts on $\fg(\delta\mapsto 0,\, \alpha_{3} \mapsto -1)$ via
    \[
        \dynkin[labels*={1,,}] B{*.**},
    \]
    hence $N_{C}$ is an $L$-spherical variety of rank 4.
    
    \item Suppose that $X = D_{n}/P_{2}$ with $\alpha = \alpha_{1}$.
    If $n= 5$, then $[\fr,\,\fr] = \mathfrak{sl}_{2}(\alpha_{4})\oplus \mathfrak{sl}_{2}(\alpha_{5})$, and it acts on $\fg(\delta\mapsto 0,\, \alpha_{3}\mapsto - 1)$ via the tensor product of the standard representations.
    Hence $N_{C}$ is an $L$-spherical variety of rank 4.
    If $n \ge 6$, then the Dynkin diagram of $[\fr,\,\fr]$ is
    \[
        \dynkin[edge length=.7cm, labels*={\alpha_{4},\alpha_{n-2},\alpha_{n-1},\alpha_{n}}] D{*.***},
    \]
    and it acts on $\fg(\delta\mapsto 0,\, \alpha_{3} \mapsto -1)$ via
    \[
        \dynkin[labels*={1,,,}] D{*.***},
    \]
    \emph{i.e.} the standard representation of $\mathfrak{so}_{2(n-3)}$.
    Thus $N_{C}$ is an $L$-spherical variety of rank 4.
    
    \item Suppose that $X = D_{n}/P_{n-2}$ $\alpha = \alpha_{n}$.
    In this case, we need to consider both $\fg(\delta \mapsto 0,\, \alpha_{n-3} \mapsto -1)$ and $\fg(\delta \mapsto 0,\, \alpha_{n-1} \mapsto -1)$.
    In fact,
    \[
        \fg(\delta \mapsto 0,\, \alpha_{n-1} \mapsto -1) = \fg_{- \alpha_{n-1}}.
    \]
    If $n = 4$, then $\fg(\delta \mapsto 0,\, \alpha_{n-3} \mapsto -1) = \fg_{-\alpha_{n-3}}$, hence $N_{C}$ is an $L$-spherical variety of rank 4.
    If $n \ge 5$, then the Dynkin diagram of $[\fr,\,\fr]$ is
    \[
        \dynkin[labels*={\alpha_{1},\alpha_{n-4}}] A{*.*}
    \]
    and it acts on $\fg(\delta \mapsto 0,\, \alpha_{n-3} \mapsto -1)$ via
    \[
        \dynkin[labels*={,1}] A{*.*}.
    \]
    Therefore $N_{C}$ is an $L$-spherical variety of rank 4.
    
    \item If $X = E_{r}/P_{3}$ with $\alpha = \alpha_{1}$, then the Dynkin diagram of $[\fr,\,\fr]$ is
    \[
        \dynkin[labels*={\alpha_{2}}] A{*} \dynkin[labels*={\alpha_{5},\alpha_{r}}] A{*.*},
    \]
    and it acts on $\fg(\delta \mapsto 0,\, \alpha_{4}\mapsto -1)$ via
    \[
        \dynkin[labels*={1}] A{*} \otimes \dynkin[labels*={1,}] A{*.*},
    \]
    \emph{i.e.} the tensor product of the standard representations of $\mathfrak{sl}_{2}$ and $\mathfrak{sl}_{r-3}$.
    Hence $N_{C}$ is an $L$-spherical variety of rank 4.
    
    \item If $X = E_{r}/P_{4}$ with $\alpha = \alpha_{2}$, then we need to consider $\fg(\delta \mapsto 0,\, \alpha_{3}\mapsto -1)$ and $\fg(\delta \mapsto 0,\, \alpha_{5}\mapsto -1)$.
    The Dynkin diagram of $[\fr,\,\fr]$ is
    \[
        \dynkin[labels*={\alpha_{1}}] A{*} \dynkin[labels*={\alpha_{6},\alpha_{r}}] A{*.*},
    \]
    and it acts on $\fg(\delta \mapsto 0,\, \alpha_{3}\mapsto -1)$ and $\fg(\delta \mapsto 0,\, \alpha_{5}\mapsto -1)$ via
    \[
        \dynkin[labels*={1}] A{*} \otimes \dynkin A{*.*} \quad \text{and} \quad \dynkin A{*} \otimes \dynkin[labels*={1,}] A{*.*},
    \]
    respectively.
    Thus $N_{C}$ is an $L$-spherical variety of rank 4.
    
    \item If $X = E_{r}/P_{r-1}$ with $\alpha = \alpha_{r}$, then the Dynkin diagram of $[\fr,\,\fr]$ is
    \begin{align*}
        \dynkin[labels*={\alpha_{1}, \alpha_{3}}] A{**} \dynkin[labels*={\alpha_{2}}] A{*} & \text{ if }r=6, \\
        \dynkin[labels*={\alpha_{1}, \alpha_{3}, \alpha_{4}, \alpha_{2}}] A{****} & \text{ if }r=7, \\
        \dynkin[edge length=.7cm,labels*={\alpha_{1}, \alpha_{3},\alpha_{4},\alpha_{2},\alpha_{5}}] D{*****} & \text{ if }r=8, \\
    \end{align*}
    and it acts on $\fg(\delta\mapsto0,\,\alpha_{r-2}\mapsto-1)$ via
    \begin{align*}
        \dynkin[labels*={, 1}] A{**} \otimes \dynkin[labels*={1}] A{*} & \text{ if }r=6, \\
        \dynkin[labels*={,, 1,}] A{****} & \text{ if }r=7, \\
        \dynkin[edge length=.7cm,labels*={,,,,1}] D{*****} & \text{ if }r=8. \\
    \end{align*}
    Observe that when $r=8$, the associated representation is the spin representation of $\mathfrak{so}_{10}$.
    Thus $N_{C}$ is an $L$-spherical variety of rank 4.
    
    \item If $X = F_{4}/P_{2}$ with $\alpha = \alpha_{1}$, then $[\fr,\,\fr] = \mathfrak{sl}_{2}(\alpha_{4})$ acts on $\fg(\delta \mapsto 0,\, \alpha_{3}\mapsto -1)$ via the standard representation.
    Thus $N_{C}$ is an $L$-spherical variety of rank 3.
\end{itemize}

\subsection{Remaining cases} \label{subsection:remaining-cases}

In this section, we consider the remaining cases for $\delta$ long: $|N(\delta)| \ge 2$ and any long element in $N(\delta)$ is not an end node.
In each case, we choose $\alpha \in N(\delta)$ so that $\alpha$ is long and $|N(\alpha)| = 2$.
More specifically, $\alpha$ is chosen as follows:
\begin{itemize}
    \item $X= A_{n}/P_{p}$ and $\alpha = \alpha_{p-1}$ ($3 \le p \le n-2$).
    \item $X= B_{n}/P_{p}$ and $\alpha = \alpha_{p-1}$ ($3 \le p \le n-1$).
    \item $X= D_{n}/P_{p}$ and $\alpha = \alpha_{p-1}$ ($3 \le p \le n-3$).
    \item $X= E_{r}/P_{5}$ and $\alpha = \alpha_{6}$ ($7 \le p \le 8$).
    \item $X= E_{8}/P_{6}$ and $\alpha = \alpha_{7}$.
\end{itemize}
Observe that $N(\delta) = \{\alpha,\, \beta\}$ and $N(\alpha) = \{\delta,\,\gamma\}$ for some $\beta\not=\alpha$ and $\gamma\not= \delta$.
Then $\ov{Q} = P_{\alpha,\,\beta}$ and $Q = P_{\alpha,\,\beta,\,\gamma}$.
By Proposition \ref{prop:coro-Kostant},
\[
    \ov{\fq}/\fq \simeq \bigoplus_{k < 0} \fg(\alpha\mapsto 0,\, \beta \mapsto 0,\, \gamma \mapsto k)
\]
as $L$-representations.
By our choice of $\alpha$, we see that $\Delta \setminus \{\alpha\}$ generates a semi-simple algebra such that the simple factor containing $\gamma$ is of type $A$.
Thus
\[
    \ov{\fq}/\fq \simeq \fg(\alpha\mapsto 0,\, \beta \mapsto 0,\, \gamma \mapsto -1).
\]
On the other hand, by Corollary \ref{coro:trivial-sl2} and by the choice of $\alpha$,
\[
    H^{0}(\ov{C},\, N_{\ov{C}/X} \otimes N^{*}_{\ov{C}/\Pi})/ \C_{0} \simeq (\fg(\delta \mapsto 0,\, \beta \mapsto - 1) \otimes \C_{\alpha}) \oplus \fg(\delta \mapsto 0,\, \alpha \mapsto 0,\, \gamma \mapsto -1)
\]
as $\fl$-representations.
By Corollary \ref{coro:sym2-normal-of-line-in-plane} and Lemma \ref{lemma:computing-normal-sp},
\[
    N_{C} \simeq (\C_{-(\delta+2\alpha)} \otimes \mathfrak{sl}_{2}(\delta)) \oplus (\fg(\delta \mapsto 0,\, \beta \mapsto - 1) \otimes \C_{\alpha}).
\]
As before, $\fz(\fl)$ acts via linearly independent characters $-(\delta + 2 \alpha)$ and $-\beta + \alpha$. We conclude by computing the second term and comparing it with Table \ref{table-spherical} case by case.
\begin{itemize}
    \item
    If $X= A_{n}/P_{p}$ and $\alpha = \alpha_{p-1}$, then the Dynkin diagram of $[\fr,\,\fr]$ is
    \[
        \dynkin[labels*={\alpha_{1},\alpha_{p-3}}]A{*.*} \dynkin[labels*={\alpha_{p+2},\alpha_{n}}]A{*.*},
    \]
    and it acts on $\fg(\delta \mapsto 0,\, \beta \mapsto - 1)$ via
    \[
        \dynkin A{*.*} \otimes \dynkin[labels*={1,}]A{*.*}.
    \]
    Hence $N_{C}$ is an $L$-spherical variety of rank 3.
    
    \item
    Suppose that $X= B_{n}/P_{p}$ and $\alpha = \alpha_{p-1}$.
    If $p=n-1$, then $\fg(\delta \mapsto 0,\, \beta \mapsto - 1) = \fg_{-\alpha_{n}}$, hence $N_{C}$ is an $L$-spherical variety of rank 3.
    If $p < n-1$, then the Dynkin diagram of $[\fr,\,\fr]$ is
    \begin{align*}
        \dynkin[labels*={\alpha_{1},\alpha_{n-5}}]A{*.*} \dynkin[labels*={\alpha_{n}}]A{*} & \text{ if }p=n-2, \\
        \dynkin[labels*={\alpha_{1},\alpha_{p-3}}]A{*.*} \dynkin[labels*={\alpha_{p+2},,\alpha_{n}}]B{*.**}, & \text{ if }p<n-2,
    \end{align*}
    and it acts on $\fg(\delta \mapsto 0,\, \beta \mapsto - 1)$ via
    \begin{align*}
        \dynkin[labels*={,}]A{*.*} \otimes \dynkin[labels*={2}]A{*} & \text{ if }p=n-2, \\
        \dynkin[labels*={,}]A{*.*} \otimes \dynkin[labels*={1,,}]B{*.**} & \text{ if }p<n-2.
    \end{align*}
    Hence $N_{C}$ is an $L$-spherical variety of 4.
    
    \item
    If $X= D_{n}/P_{p}$ and $\alpha = \alpha_{p-1}$, then the Dynkin diagram of $[\fr,\,\fr]$ is
    \begin{align*}
        \dynkin[labels*={\alpha_{1},\alpha_{n-6}}]A{*.*} \dynkin[labels*={\alpha_{n-1}}] A{*} \dynkin[labels*={\alpha_{n}}] A{*} & \text{ if }p=n-3, \\
        \dynkin[labels*={\alpha_{1},\alpha_{p-3}}]A{*.*} \dynkin[labels*={\alpha_{p+2},,\alpha_{n-1},\alpha_{n}}] D{*.***} & \text{ if }p<n-3,
    \end{align*}
    and it acts on $\fg(\delta \mapsto 0,\, \beta \mapsto - 1)$ via
    \begin{align*}
        \dynkin A{*.*} \otimes \dynkin[labels*={1}] A{*} \otimes \dynkin[labels*={1}] A{*} & \text{ if }p=n-3, \\
        \dynkin A{*.*} \otimes \dynkin[labels*={1,,,}] D{*.***} & \text{ if }p<n-3.
    \end{align*}
    Hence $N_{C}$ is an $L$-spherical variety of rank 4.
    
    \item
    If $X= E_{r}/P_{5}$ and $\alpha = \alpha_{6}$, then the Dynkin diagram of $[\fr,\,\fr]$ is
    \begin{align*}
        \dynkin[labels*={\alpha_{1},\alpha_{3}}]A{**} \dynkin[labels*={\alpha_{2}}] A{*} & \text{ if }r=7, \\
        \dynkin[labels*={\alpha_{1},\alpha_{3}}]A{**} \dynkin[labels*={\alpha_{2}}] A{*} \dynkin[labels*={\alpha_{8}}] A{*} & \text{ if }r=8,
    \end{align*}
    and it acts on $\fg(\delta \mapsto 0,\, \beta \mapsto - 1)$ via
    \begin{align*}
        \dynkin[labels*={,1}]A{**} \otimes \dynkin[labels*={1}] A{*} & \text{ if }r=7, \\
        \dynkin[labels*={,1}]A{**} \otimes \dynkin[labels*={1}] A{*} \otimes \dynkin A{*} & \text{ if }r=8.
    \end{align*}
    Hence $N_{C}$ is an $L$-spherical variety of rank 4.
    
    \item
    If $X= E_{8}/P_{6}$ and $\alpha = \alpha_{7}$, then the Dynkin diagram of $[\fr,\,\fr]$ is
    \[
        \dynkin[labels*={\alpha_{1},\alpha_{3}, \alpha_{4}, \alpha_{2}}]A{****}
    \]
    and it acts on $\fg(\delta \mapsto 0,\, \beta \mapsto - 1)$ via
    \[
        \dynkin[labels*={,,1, }]A{****}.
    \]
    Hence $N_{C}$ is an $L$-spherical variety of rank 4.
\end{itemize}

Now a representation theoretic proof of Theorem \ref{thm:main} for $\delta$ long is completed.
Note that in the process of the proof, since $\rank_{L}(N_{C}) = \rank_{G}(\Hilb_{2}(X))$ by Lemma \ref{lemma:loc-str-spherical}, we obtain the following corollary.

\begin{cor} \label{cor:rank}
    For $\delta$ long, $\Hilb_{2}(X)$ is a smooth $G$-spherical of rank $r$ with 
    \begin{enumerate}
        \item $r = 2$ if $X = A_{n}/P_{1}$, $A_{n}/P_{n}$, $C_{n}/P_{n}$ (with $n \ge 2$), $B_{2}/P_{1}$ or $G_{2}/P_{2}$,
        \item $r = 3$ if $X=A_{n}/P_{p}$ (with $2 \le p \le n-1$), $B_{n}/P_{1}$ (with $n \ge 3$), $B_{n}/P_{n-1}$, $D_{n}/P_{1}$, $D_{n}/P_{n-1}$, $D_{n}/P_{n}$ (with $n \ge 4$), $E_{6}/P_{1}$, $E_{6}/P_{6}$, $E_{7}/P_{7}$, or $F_{4}/P_{2}$,
        \item $r = 4$ otherwise.
    \end{enumerate}
%        If $\delta$ is long, then $\Hilb_{2}(X)$ is a smooth $G$-spherical variety of rank equal to
%    \begin{itemize}
%        \item 2 if $X = A_{n}/P_{1}$ ($n \ge 2$), $A_{n}/P_{n}$ ($n \ge 2$), $C_{n}/P_{n}$ ($n \ge 2$), $B_{2}/P_{1}$ or $G_{2}/P_{2}$,
%        \item 3 if $X=A_{n}/P_{p}$ ($2 \le p \le n-1$), $B_{n}/P_{1}$ ($n \ge 3$), $B_{n}/P_{n-1}$, $D_{n}/P_{1}$ ($n \ge 4$), $D_{n}/P_{n-1}$ ($n \ge 4$), $D_{n}/P_{n}$ ($n %\ge 4$), $E_{6}/P_{1}$, $E_{6}/P_{6}$, $E_{7}/P_{7}$, or $F_{4}/P_{2}$, and
%        \item 4 otherwise.
%    \end{itemize}
\end{cor}

Note that for $X$ cominuscule, then $\rank(\Hilb_{2}(X)) \leq 3$.

%In particular, if $X$ is an irreducible hermitian symmetric space, then the rank of $\Hilb_{2}(X)$ is at most 3.

\subsection{The case of short roots} \label{subsection:short-roots}

In this section, we give a representation theoretic proof of Theorem \ref{thm:main} for $\delta$ short \emph{i.e.} we show that $N_{C}$ is not $L$-spherical. We could try to fully compute the representation $N_C$. To simplify our strategy of the proof is to identify a subrepresentation of $N_C$ which is not spherical. We discuss two cases: \begin{enumerate}
\item if there exists a long root $\alpha \in N(\delta)$, it suffices to consider $\Hilb_2(\Pi)$, %the spaces of conics in $\Pi$,
%\item $\Sym^{2} H^{0}(\ov{C},\, N_{\ov{C}/\Pi}) \subset N_C$ is a non-spherical representation.
\item otherwise, we need to let $\Pi$ vary in a maximal projective subspace of $X$.
\end{enumerate}

Assume first that there exists a long root $\alpha \in N(\delta)$ and consider $\Hilb_2(\Pi)$ the space of conics contained in $\Pi$. Then $\Hilb_2(\Pi)$ contains as a subspace the space $\Aut(\Pi).C$ of double conics and $N_C$ contains the normal space 
$$\bar N_C = N_{\Aut(\Pi).C/\Hilb_2(\Pi),C}.$$
The same computations as in Lemma \ref{lemma:computing-normal-sp} (with $G = \Aut(\Pi)$) show that $\bar N_C \simeq \Sym^{2} H^{0}(\ov{C},\, N_{\ov{C}/\Pi})$. We get the following result.

\begin{prop}
\label{prop:near-long}
 If there exists a long root $\alpha \in N(\delta)$, then $N_C$ is not $L$-spherical.
\end{prop}

\begin{proof}
By the previous discussion, we have that $\Sym^{2} H^{0}(\ov{C},\, N_{\ov{C}/\Pi}) \simeq \bar N_C$ is a subrepresentation of $N_C$. The result thus follows from Corollary \ref{coro:sym2-not-normal}.
\end{proof}

\begin{remark}
Proposition \ref{prop:near-long}  and Lemma \ref{lemma:loc-str-spherical} give a proof of Theorem \ref{thm:main} in the cases $B_n/P_n$, $C_n/P_{n-1}$, $F_4/P_3$ and $G_2/P_1$.
\end{remark}

The representation $\Sym^{2} H^{0}(\ov{C},\, N_{\ov{C}/\Pi})$ used in the previous proof is always a subrepresentation of $N_C$, however as explained in Corollary \ref{coro:sym2-not-normal}, this representation is not spherical only if $|\alpha| > |\delta|$. In the left cases, we have $|\alpha| = |\delta|$ for all $\alpha \in N(\delta)$ so that we cannot use the previous argument. We therefore need to identify a bigger subrepresentation of $N_C$ and for this we consider a family of planes containing $\Pi$. Let $\alpha \in N(\delta)$ be the root contained in the non simply-laced connected component of the Dynkin diagram obtained by removing $\delta$.

\begin{example}
Note that the cases we consider are $C_n/P_k$ for $k \in [1,n-2]$ and $F_4/P_4$. With notation as in \cite{bourbaki:elements*78}, we have
\begin{itemize}
\item $\alpha = \alpha_{k+1}$ for $C_n/P_k$ and 
\item $\alpha = \alpha_3$ for $F_4/P_4$.
\end{itemize}
\end{example}

We consider a family of maximal projective subspaces contained in $X$ obtained by selecting the largest Dynkin subdiagram of the Dynkin diagram of $G$ which is of type $C_m$ and such that $\delta$ is the unique short simple end-root. Explicitly, we consider a subset $\Sigma \subset \Delta$ of simple roots. Then the family of projective subspaces is $G/P_\Sigma$ and each such subspace is obtained as the fiber of $G/P_{\Sigma\cup\{\delta\}} \to G/P_\Sigma$ identified with its image in $X = G/P_\delta$. With notation as in \cite{bourbaki:elements*78}, we have
\begin{itemize}
\item $\Sigma = \{\alpha_{k-1}\}$ for $C_n/P_k$. In this case the maximal projective space is $\bP^{2(n-k)+1}$ and the family of these subspaces is $G/P_{k-1}$. 
\item $\Sigma = \{\alpha_1\}$ for $F_4/P_4$.  In this case the maximal projective space is $\bP^{7}$ and the family of these subspaces is $F_4/P_{1}$. 
\end{itemize}

The projective space over the $B$-fixed point in $G/P_\Sigma$ is $P_\Sigma/P_{\Sigma \cup \{\delta\}} = \bP(W_\Sigma)$ for some $P_\Sigma$-representation $W_\Sigma$. Let $\Gr(3,W_\Sigma)$ be the grassmannian of planes in $\bP(W_\Sigma)$ and set $\cP = G \times^{P_\Sigma} \Gr(3,W_\Sigma)$ be the associated familly of planes in $X$. We have a natural $G$-equivariant morphism $\cP \to G/P_\Sigma$. Note that $\Pi \subset \bP(W_\Sigma)$ is a $B$-stable plane. Its $P_\Sigma$-orbit $P_\Sigma.\Pi \subset \Gr(3,W_\Sigma)$ is the (unique) closed $P_\Sigma$-orbit. We get that $G.\Pi \subset \cP$ is the (unique) closed $G$-orbit. Finally, let $\cE$ be the universal subbundle over $\Gr(3,W_\Sigma)$. Denote also by $\cE$ its pull back to $\cP$. Then $\cF = \bP (\Sym^{2} \cE^{*}|_{\cP})$ is the space of conics in $X$ contained in a plane of the family $\cP$. In particular, since $\Pi \in \cP$, we have $C \in \cF$. Define $\tilde N_C = N_{G.C/\cF,C}$. The above discussion shows that $\tilde N_C$ is a subrepresentation of $N_C$, we even have the sequence of subrepresentations $\bar N_C \subset \tilde N_C \subset N_C$. The following result concludes the proof of Theorem  \ref{thm:main}.

\begin{prop}
\label{prop:near-short}
With the above notations, we have $\tilde N_C$ is not $L$-spherical.
\end{prop}

\begin{proof}
The same computations as in Lemma \ref{lemma:computing-normal-sp} or the discussion before Proposition \ref{prop:near-long} show that $\bar N_C \simeq \Sym^{2} H^{0}(\ov{C},\, N_{\ov{C}/\Pi})$. We need to compute the representation $\tilde N_C/\bar N_C$ which is isomorphic to $N_{P_\Sigma.\Pi/\Gr(3,W_\Sigma),\Pi}$ and because $P_{\Sigma}.\Pi$ is the isotropic grassmannian for the form defined by $P_\Sigma$ this space is isomorphic to $\Lambda^2 \Pi^{*}$. Now by Corollary \ref{coro:sym2-normal-of-line-in-plane}, we have $\bar N_C = \mathfrak{sl}_2(\delta) \otimes \C_{\chi_1}$ as $L$-representation for some character $\chi_1$ of $T$. Since $L$ stabilises the line $\bar C \subset \Pi$ and acts on $\bar C$ via $\SL_2(\delta)$ and $T$, we get that the isomorphism of $L$-representations $\Lambda^2\Pi^{*} \simeq \left(V_{\mathfrak{sl}_2(\delta)} \otimes \C_{\chi_2}\right) \oplus \C_{\chi_3}$ for some characters $\chi_2$ and $\chi_3$. The group $L$ therefore acts on 
$$\tilde N_C \simeq \left(\mathfrak{sl}_2(\delta) \otimes \C_{\chi_1}\right) \oplus \left(V_{\mathfrak{sl}_2(\delta)} \otimes \C_{\chi_2}\right) \oplus \C_{\chi_3}$$ 
via $\SL_2(\delta) \times \bG_m^3$. However, since a Borel subgroup of $\SL_2(\delta) \times \bG_m^3$ is of dimension 5, the representation $\tilde N_C$ is not $L$-spherical.
\end{proof}

%%%%%%%%%%%%%%%%%%%%%%%%%%%%%%%%%%%%%%%%%%%%%%%%%%%%%%

% Bibliography
% A .bib file can be used by adding the line "\def\mybibfile{file.bib}" to the
% file equivrigid.conf. Do not add equivrigid.conf to the repository.
% \ifdefined\mybibfile
% \bibliography{\mybibfile}
% \else

% \fi
% \bibliographystyle{halpha}

\end{document}